\documentclass[11pt,letterpaper]{article}

\usepackage{amsmath,latexsym,amssymb,ifthen}
\usepackage{amsthm}
\usepackage{xspace}
\usepackage{ulem}
\usepackage{algorithm}
\usepackage[noend]{algorithmic}
\usepackage{tikz}
\usepackage{float}
\usepackage{paralist}
\usepackage{hyperref}
\usepackage[alphabetic,abbrev]{amsrefs}

\renewcommand{\emph}[1]{\textit{#1}}

\usepackage{color}\usepackage{graphicx}

\newcommand{\gbsxx}[1]{\relax}

\definecolor{brown}{cmyk}{0, 0.72, 1, 0.45}
\definecolor{grey}{gray}{0.5}

\newboolean{draft}
\setboolean{draft}{true}

\def\cT{\mathcal{T}}

\newcommand{\set}[1]{\left\{#1\right\}}
\newcommand{\card}[1]{\left|#1\right|}
\newcommand{\abs}[1]{\left|#1\right|}

\def\cQ{\mathcal{Q}}

\def\ii_(#1,#2){i_{#1}^{#2}}

\def\qs{q.s.}
\def\Qs{Q.s.}
\def\bA{\bar{A}}
\def\bB{\bar{B}}
\def\bC{\bar{C}}

\def\a{\alpha}
\def\b{\beta}
\def\d{\delta}

\def\e{\varepsilon}
\def\f{\phi}
\def\F{\Phi}

\def\z{\zeta}

\def\l{\lambda}

\def\p{\pi}

\def\r{\rho}

\def\s{\sigma}

\def\t{\tau}
\def\om{\omega}

\def\x{\xi}

\newcommand{\ceil}[1]{\left\lceil #1 \right\rceil}
\newcommand{\parens}[1]{\left( #1 \right)}
\newcommand{\floor}[1]{\left\lfloor #1 \right\rfloor}

\newcommand{\whp}{{w.h.p.}\xspace}

\newcommand{\proj}{\operatorname{proj}}

\newcommand{\triple}[1]{(#1,#1,#1)}

\def\cE{\mathcal{E}}

\newcommand{\OTR}[2]{((\SS{#1}{#2})^*,\SS{#1}{#2},X)}
\newcommand{\brac}[1]{\left( #1 \right)}

\newcommand{\expect}{\operatorname{\bf E}}
\def\E{\expect}
\renewcommand{\Pr}{\operatorname{\bf Pr}}
\newcommand\bfrac[2]{\left(\frac{#1}{#2}\right)}

\def\bC{\bar{C}}

\newtheorem{theorem}{Theorem}
\newtheorem{lemma}{Lemma}
\newtheorem{definition}[lemma]{Definition}
\newtheorem{property}[lemma]{Property}

\newtheorem{remthm}[lemma]{Remark}
\newtheorem{claim}[lemma]{Claim}

\newenvironment{remark}{\begin{remthm} }{\end{remthm}}%
\newcounter{thmtemp}

\newcommand{\nospace}[1]{}

\def\path{\operatorname{PATH}}

\def\BDTS{\texttt{BDAPTA}}
\newcommand{\BDTSx}[1]{\BDTS$(#1)$}
\def\expdist{\ensuremath{\operatorname{Exp}}}
\def\EX{\expdist(1)}

\newcommand{\Bin}{\ensuremath{\operatorname{Bin}}}

\def\V{{\bf Var}}
\newcommand{\beq}[1]{\begin{equation}\label{#1}}
\def\eeq{\end{equation}}
\newcommand{\ZPl}[1]{Z^{P}_{#1,n}}
\newcommand{\ZAx}[1]{Z^{A}_{#1,n}}

\newcommand{\Xijk}{X_{i,j,k}\:}
\newcommand{\Mijk}{M_{i,j,k}}
\newcommand{\CC}{M_{i,j,k}}

\renewcommand{\ss}[2]{s^{#1}_{#2}}
\renewcommand{\SS}[2]{{S^{#1}_{#2}}}

\newcommand{\SSstar}[2]{{(S^\star)}^{#1}_{#2}}
\newcommand{\Triples}{\textsc{Triples}\xspace}
\newcommand{\SetConstruct}{\textsc{SetConstruct}\xspace}
\newcommand{\PoisonPropagate}{\textsc{PoisonPropagate}\xspace}
\newcommand{\MakeTree}{\textsc{MakeTree}\xspace}
\newcommand{\GreedyPhase}{\textsc{GreedyPhase}\xspace}
\newcommand{\MainPhase}{\textsc{MainPhase}\xspace}
\newcommand{\FinalPhase}{\textsc{FinalPhase}\xspace}
\newcommand{\fail}{\textsc{Fail}\xspace}
\newcommand{\Pstar}{P^{\star}}

\newcommand{\Sstar}{S^{\star}}

\newcommand{\Jstar}{J^{\star}}

\newcommand{\Ustar}{U^{\star}}

\newcommand{\sz}{\sigma}
\newcommand{\xcard}{n'} 
\newcommand{\nt}{{19/20}} % was originally 9/10
\newcommand{\xcardl}{n-2 n^{\nt}}
\newcommand{\mainalgsec}[1]{\medskip \noindent \textbf{#1}}
\newcommand{\outlist}{\operatorname{AugOut}}

\newcommand{\kpell}{{\dd+1-\ell}}

\newcommand{\rate}{{(1 \pm \d)}}
\newcommand{\ratr}{{(1 \pm \d)}}
\newcommand{\ratp}{{(1 + \delta)}}
\newcommand{\kp}{{\dd+1}}
\newcommand{\ellp}{{\ell+1}}

\newcommand{\argmax}{\operatorname{argmax}}
\newcommand{\sspar}[2]{ (\ss #1 #2, \ss{#1+1}{2 #2 -1}, \ss{#1+1}{2 #2}) }
\newcommand{\bitsi}[1]{ b_1 \ldots b_{#1}}
\newcommand{\emptystring}{\$\xspace}
\newcommand{\pii}[1]{\p(i_{#1})}
\newcommand{\piip}[1]{\p'(i_{#1})}

\newcommand{\ptripp}[3]{(\piip{#1}, \piip{#2}, \piip{#3})}

\def\DD{D}
\def\dd{d}
\def\ddp{{d+1}}
\def\pp{{\rho}}
\def\tk{\eta}

\newcommand{\sgeq}{\succsim}
\def\err{\delta}

\newcommand{\outtrip}[1]{ (#1^{\star}, #1, X) }
\newcommand{\outnode}{(\ss \ell m, \ss\ellp{2m-1}, \ss\ellp{2m})}

\newcommand{\xx}{{n}}
\newcommand{\outtripp}[2]{(\SSstar{#1}{#2}, \SS{#1}{#2}, X)}
\newlength{\mysep}
\newcommand{\oneto}[1]{\set{1,\ldots,#1}}
\newcommand{\linefrac}[2]{#1 / #2}
\newcommand{\four}{4\xspace}
\newcommand{\fourm}{3\xspace}
\newcommand{\fourp}{5\xspace}
\newcommand{\FAILIF}[1]{\STATE \textbf{if} #1 \textbf{then} \fail}
\newcommand{\LP}{{\operatorname{LP}}}
\newcommand{\kk}{k}

\begin{document}
% \belowdisplayskip=12pt plus 3pt minus 9pt
% \abovedisplayshortskip=7pt plus 3pt minus 4pt
% \belowdisplayshortskip=7pt plus 3pt minus 4pt
\abovedisplayshortskip=5pt plus 3pt minus 4pt
\belowdisplayshortskip=5pt plus 3pt minus 4pt

\title
{\vspace*{-.5in} 
Efficient algorithms for three-dimensional \\ 
axial and planar
random assignment problems}

\author{Alan Frieze%
\thanks{Research supported by NSF grant CCF-1013110.
Department of Mathematical Sciences,
Carnegie Mellon University, Pittsburgh PA 15213, USA.
\hbox{e-mail}~{\small\texttt{alan@random.math.cmu.edu}}}
\and
Gregory B.~Sorkin\thanks{%
Department of Management, 
London School of Economics,
London WC2A 2AE, England.
\hbox{e-mail}~{\small\texttt{g.b.sorkin@lse.ac.uk}}
The work was done in part when the author was at the
Department of Mathematical Sciences,
IBM T.J.~Watson Research Center, Yorktown Heights NY 10598, USA.
}}

\maketitle
\markboth{}{}
\thispagestyle{empty}

\begin{abstract}
Beautiful formulas are known for the expected cost of
random two-dimensional assignment problems,
but in higher dimensions even the scaling is not known.
In three dimensions and above, the problem has natural ``Axial''
and ``Planar'' versions, both of which are NP-hard.
For 3-dimensional Axial random assignment instances of size $n$,
the cost scales as $\Omega(1/n)$,
and a main result of the present paper is
a linear-time algorithm that, with high probability, 
finds a solution of cost $O(n^{-1+o(1)})$.
For 3-dimensional Planar assignment,
the lower bound is $\Omega(n)$,
and we give a new efficient matching-based algorithm
that with high probability returns a solution with cost $O(n \log n)$.
\end{abstract}

\section{Introduction}
An instance of the (two-dimensional) assignment problem
may be thought of as an $n \times n$ cost array $M_{i,j}$,
a candidate solution is a permutation $\pi \colon [n] \mapsto [n]$,
its cost is
$\sum_{i=1}^n M_{i,\pi(i)}$,
and an optimal solution is one minimizing the cost.
If the cost matrix represents, for example,
the costs of assigning various jobs $i$ to machines $j$,
where each machine can accommodate only one job,
then the problem's solution represents the cheapest
way of assigning the jobs to machines.
It may equivalently be formulated as an integer linear program,
minimizing the sum of selected elements consistent with the selection
of exactly one element from each row and from each column, i.e.,
minimizing
$\sum_{i,j} M_{i,j}X_{i,j} $
where $X_{i,j} \in \set{0,1}$,
$(\forall i) \; \sum_{j} X_{i,j}=1$
and $(\forall j) \; \sum_{i} X_{i,j}=1$.
This is a network flow problem,
thus its linear relaxation with $X_{i,j} \in [0,1]$
has integer extreme points,
and the problem may be solved in polynomial time.

The \emph{random assignment problem}, in its most popular form,
is the case when 
the entries of the cost matrix $C$ are i.i.d.\ \EX\ random variables
(independent, identically distributed 
exponential random variables with parameter~1).
Since the problem can be solved in polynomial time, 
the focus for the random case is on
the cost's expectation as a function of $n$,
$$f(n) = \E\Big[ \min_{\pi} \sum_{i=1}^n M_{i,\pi(i)} \Big]
= \E \Big[ \min_{X_{i,j}} \sum_{i,j}  M_{i,j} X_{i,j}\Big] $$
with $X_{i,j} \in \set{0,1}$ subject to
the row and column constraints above.
This problem has received a great deal of study over several decades.
It was considered from an operations research perspective
in the 1960s \cite{Donath69},
an asymptotic conjecture $f(n) \rightarrow \pi^2/6= \zeta(2)$
was formulated by statistical physicists M\'ezard and Parisi in the 1980s
based on the mathematically sophisticated but non-rigorous ``replica method''
\cite{Mezard85,Mezard87},
an exact conjecture $f(n) = \sum_{i=1}^n 1/i^2$
was hazarded by Parisi in the late 1990s \cite{Parisi98}
a generalization
to partial matchings and non-square matrices
was made by Coppersmith and Sorkin \cite{CS99},
the M\'ezard--Parisi
conjecture was proved by Aldous in a pair of papers in 1992 and 2001
\cite{Aldous92,Aldous01},
and the Coppersmith--Sorkin conjecture was proved simultaneously in 2004
by two papers using two very different methods,
by Nair, Prabhakar and Sharma \cite{NaPrSh}, 
and Linusson and W\"astlund \cite{Wastlund04}.
A further generalisation of these conjectures was made 
by Buck, Chan and Robbins in 2002 \cite{BCR} 
and proved by W\"astlund in 2005 \cite{Wastlund05}.
The study of 
other aspects of the random assignment problem and 
related problems is ongoing, 
for example by W\"astlund in \cite{Wastlund09}.

In higher dimensions there are two natural generalizations
of the assignment problem.
For example in three dimensions,
the Axial assignment problem is,
given an $n \times n \times n$ matrix
(or ``tensor'' or ``array'') $C$,
to find a solution $\Xijk$
minimizing
$\sum_{i,j,k}  \Mijk \Xijk$
where $\Xijk \in \set{0,1}$
and there is one selected value per ``plane'' of the array, 
of which there are three types, 
1-, 2-, and 3-planes according to which coordinate is fixed:
\beq{AP}
(\forall i) \; \sum_{j,k} \Xijk=1 , \quad
  (\forall j) \; \sum_{i,k} \Xijk=1 , \quad
  (\forall k) \; \sum_{i,j} \Xijk=1 . 
\eeq
Equivalently, the Axial problem is to determine
$ \min_{\pi,\sigma} \sum_{i=1}^n M_{i,\pi(i),\sigma(i)} $,
the minimum taken over a pair of permutations.
The Planar three-dimensional assignment problem is similar but
with one selected value per ``line'' of the array,
with three types of lines:
\beq{PP}
(\forall i,j) \; \sum_{k} \Xijk=1 , \quad
  (\forall j,k) \; \sum_{i} \Xijk=1 , \quad
  (\forall i,k) \; \sum_{j} \Xijk=1 . 
\eeq
The generalizations to higher dimensions are clear.
In three dimensions and higher, the Axial and Planar assignment problems
are both NP-hard.
The Axial case in three dimensions was one of the original problems
listed by Karp \cite{Karp}.
The complexity of the Planar problem was established by Frieze \cite{Frieze3d}.

The \emph{multi-dimensional random assignment problem}
we consider here is the case when the entries of the cost matrix
are i.i.d.\ \EX\ random variables.
In this random setting, there are two natural questions.
First, are there polynomial-time algorithms that
find optimal or near-optimal solutions \whp?%
\footnote{A sequence of events 
$\cE_n,n\geq 0$ is said to occur \emph{with high probability} (\whp)
if $\lim_{n\to\infty}\Pr(\cE_n)=1$.
}
Second, what is the expected cost of a minimum assignment?
A random two-dimensional assignment instance has limiting expected cost
$\zeta(2)$,
and Frieze showed that
the expected cost of a minimum spanning tree in the complete graph with
random \EX\ edge weights tends to $\zeta(3)$ \cite{FriezeMST},
so it is tantalizing to wonder if there might be similarly beautiful
expressions for the expected cost in
multi-dimensional versions of the random assignment problem.
However, we do not even know how the cost scales with $n$.

Some of the characteristics and applications of these problems are discussed
in a recent book by Burkard, Dell'Amico and Martello \cite{BDM}.
Little
is known about the probabilistic behavior of these problems for dimension at least three,
and even less about polynomial-time algorithms
for constructing good solutions. First consider the Axial problem with
constraints \eqref{AP}.
Grundel, Krokhmal and Pardalos \cite{GKP} 
replace the \EX\ assumption with more general distributions.
Where $F(x)$ is the probability that $\Mijk \leq x$,
their result most relevant to our discussion is that if $F^{-1}(x)=O(x^\b)$ for some $\b>0$ as $x\to 0^+$
then the minimum value $\ZAx{d}\to 0$ \whp\
It is not difficult to prove $\ZAx{d}\to
0$ \whp\ for $\Mijk =\EX$
(see Remark~\ref{GreedyNote}),
but we are interested in the precise convergence rate.
Kravtsov \cite{Krav1} describes a class of greedy algorithms that work well 
if \EX\ is replaced by
uniform over $\set{1,2,\ldots,n^\a}$ where $\a<1$ depends on the particular algorithm.
The lower bound of 1 means that the minimum is at least $n$ and this is not a difficult target, asymptotically.

The Planar problem with constraints \eqref{PP} was considered by Dichkovskaya and Kravtsov \cite{DK}. Here they
discuss a ``greedy'' algorithm similar to that proposed by us. Their analysis is quite different and their
distribution is 
(as with \cite{Krav1}) 
uniform over $\set{1,2,\ldots,n^\a}$ where $\a<1$. This makes the minimum
at least $n^2$ and again this is not a difficult target, asymptotically.

Statistical physicists Martin, M\'ezard and Rivoire
conjecture in \cite{Mezard05}
that the Axial problem has an asymptotic expected cost of $c/n$,
based on ``cavity'' calculations; see also \cite{Mezard04}.
However, there is no nice constant like $\zeta(2)$ predicted,
and no certainty that the conjectured scaling is correct.

\section{Summary of results, methods, and limitations, and outline} \label{secsummary}
\subsection{Axial assignment}
For the Axial ${\DD}$-dimensional assignment problem,
there is an easy lower bound of $\Omega(1/n^{{\DD}-2})$ on the expected cost
(see Theorem~\ref{th1}).
Our main result (Theorem~\ref{th2}) is for the case ${\DD}=3$. 
Here we give an algorithm that
runs in time linear in $n$ 
and yields \whp\ a solution of cost $O(1/n^{1-o(1)})$, 
an $n^{o(1)}$ approximation to the best possible. 
Not only is this the first nearly tight upper bound obtained algorithmically,
it is the only good bound except for one (see Theorem~\ref{th1})
following from a recent non-constructive result on hypergraph factors
by Johansson, Kahn and Vu \cite{JKV}.

Our algorithm may be compared with one in
\cite{CS99} for 2-dimensional assignment.
There, a bipartite matching was augmented by an
alternating path of bounded length,
with care taken in regard to ``conditioning'' of the cost matrix.
Here, partial assignments are augmented with
a ``bounded-depth alternating-path tree'', 
a tree in which 
a newly added element displaces two previously selected elements,
those two elements are replaced in a way displacing four selected elements,
and so on, until all the displaced elements are replaced by
elements in a non-conflicting, ``unassigned'' set
(see Figure~\ref{figfig}).

\usetikzlibrary{decorations.pathmorphing}
\usetikzlibrary{positioning}
\usetikzlibrary{arrows,automata}

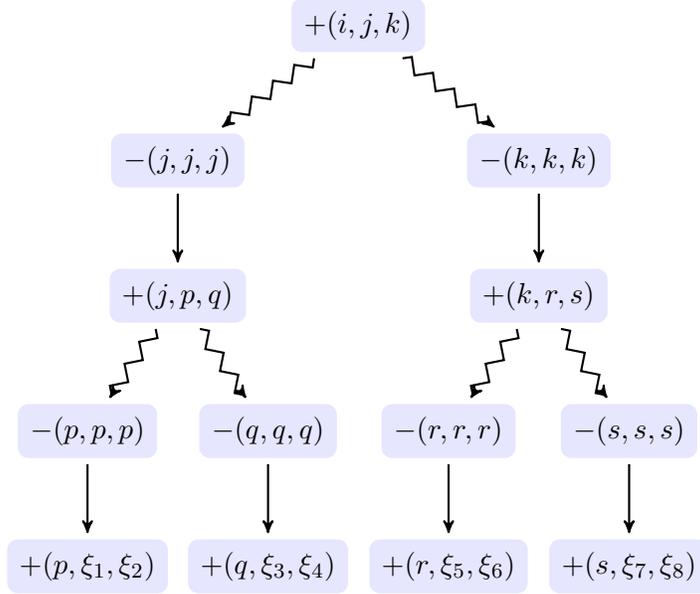
\begin{figure}[htb] 
\centering
  \begin{tikzpicture}[scale=0.6,thick]
   \tikzstyle{vertex}=
     [rounded corners,fill=blue!10,inner sep=1ex,outer sep=0.5ex]
   \tikzstyle{edge}=[->,>=stealth']
   \tikzstyle{antiedge}=[decorate,decoration=zigzag,->,>=stealth',]

    \node [vertex] (a1) at (0,0) {$+(i,j,k)$};

    \node [vertex] (b1) at (-4,-3) {$-(j,j,j)$};
    \node [vertex] (b2) at (+4,-3) {$-(k,k,k)$};
    \draw [antiedge] (a1)--(b1);
    \draw [antiedge] (a1)--(b2);

    \node [vertex] (c1) at (-4,-6) {$+(j,p,q)$};
    \node [vertex] (c2) at (+4,-6) {$+(k,r,s)$};
    \draw [edge] (b1)--(c1);
    \draw [edge] (b2)--(c2);

    \node [vertex] (d1) at (-6,-9) {$-(p,p,p)$};
    \node [vertex] (d2) at (-2,-9) {$-(q,q,q)$};
    \node [vertex] (d3) at (+2,-9) {$-(r,r,r)$};
    \node [vertex] (d4) at (+6,-9) {$-(s,s,s)$};
    \draw [antiedge] (c1)--(d1);
    \draw [antiedge] (c1)--(d2);
    \draw [antiedge] (c2)--(d3);
    \draw [antiedge] (c2)--(d4);

    \node [vertex] (e1) at (-6,-12) {$+(p,\x_1,\x_2)$};
    \node [vertex] (e2) at (-2,-12) {$+(q,\x_3,\x_4)$};
    \node [vertex] (e3) at (+2,-12) {$+(r,\x_5,\x_6)$};
    \node [vertex] (e4) at (+6,-12) {$+(s,\x_7,\x_8)$};
    \draw [edge] (d1)--(e1);
    \draw [edge] (d2)--(e2);
    \draw [edge] (d3)--(e3);
    \draw [edge] (d4)--(e4);

   \end{tikzpicture}
 \caption{\label{figfig}
  Diagram showing how a depth-2 ``augmentation'' (or alternating-path tree)
  increases an assignment's cardinality by 1.
  This is the basis of algorithm \BDTSx{2};
   see Section~\ref{MP0}.
     Without loss of generality we assume that all current assignment
     elements are of the form $\triple u$ 
     Adding new first coordinate $i$ to the partial assignment,
     using element $(i,j,k)$,
     implies deletion of previous assignment elements
     $(j,j,j)$ and $(k,k,k)$. 
     Their first coordinates $j$ and $k$ are
     then reassigned
     respectively to elements $(j,p,q)$ and $(k,r,s)$,
     displacing four more existing assignment elements $(p,p,p)$ etc. 
     Their first coordinates are finally reassigned to
     unused second and third coordinates 
     by using elements $(p,\x_1,\x_2)$ etc.
  \label{treepic}}
\end{figure}

For dimensions ${\DD} \geq 4$ the algorithm can still be applied, 
but it produces a solution whose expected cost is of order $\Omega(n^{-1})$,
just as in the 3-dimensional case;
for ${\DD} \geq 4$ this is far from the upper bound of $O(n^{2-{\DD}} \log n)$
given in Theorem \ref{th1}.

\subsection{Planar assignment}
Our second main result (Theorem~\ref{th3}),
establishes high-probability lower and upper bounds for the cost of a 
Planar 3-dimensional assignment.
A lower bound of $\Omega(n)$ comes from attending to only the 
first constraint in \eqref{PP}.
We then show that a 3-dimensional Planar assignment
consists of $n$ 2-dimensional assignments,
with constraints between them.
An upper bound is derived from a greedy algorithm that solves
the 2-dimensional assignments sequentially,
respecting the constraints between them.
These 2-dimensional assignment instances have a 
somewhat complex structure,
but a general result of Dyer, Frieze, and McDiarmid \cite{DFM}
is insensitive to the details 
and yields an upper bound of $O(n \log n)$.

As for the Axial case, though, our approach to the Planar problem
falters for dimensions ${\DD} \geq 4$.
The natural generalization is to a greedy algorithm that sequentially 
solves $n$ $({\DD}-1)$-dimensional instances,
but even for ${\DD}=4$ such an algorithm can fail,
reaching an instance that has no solution, regardless of cost.

\subsection{Structure of the paper}
Our results for Axial assignment are formally stated in Section \ref{PV}.
There, Theorem \ref{th1} gives a simple lower bound,
and a non-constructive upper bound
that is also proved easily but only by relying on some heavy machinery.
Theorem \ref{th2} gives our main Axial result, 
a constructive upper bound for 3-dimensional assignment ($D=3$)
coming from
a bounded-depth alternating-path tree algorithm \BDTS=\BDTSx{\dd},
where $2\dd$ is the depth of the search tree. 
The algorithm and its analysis are undeniably complicated.
To provide intuition, \BDTSx{2} is described and analyzed in 
Section \ref{PV2}.
The general algorithm \BDTSx{\dd} is analyzed in Section \ref{MG},
completing the proof of Theorem \ref{th2}.
Difficulties for dimensions ${\DD} > 3$ are sketched in 
Section \ref{higherAxial}.

The Planar problem is considered in Section \ref{AV}.
In Section \ref{mainPlanar}, Theorem \ref{th3}
states a simple lower bound and a constructive upper bound.
These are proved in Sections \ref{3GA1} and \ref{3GA} respectively.
Difficulties for dimensions ${\DD} > 3$ are sketched in 
Section \ref{higherPlanar}.

\section{Multi-dimensional Axial assignment results}\label{PV}
\subsection{Simple bounds} \label{simple}

\begin{theorem}\label{th1}
The minimum, $\ZAx{{\DD}}$, of the ${\DD}$-dimensional Axial random assignment problem satisfies
$$\Omega\bfrac{1}{n^{{\DD}-2}}\leq \ZAx{{\DD}}\leq O\bfrac{\log n}{n^{{\DD}-2}}\quad\whp$$
\end{theorem}

\begin{proof}
Clearly
$$\ZAx{{\DD}}\geq \sum_{i_1=1}^n\min_{i_2,\ldots,i_{\DD}}M_{i_1,\ldots,i_{\DD}}.$$
Each summand 
is distributed as $\expdist(n^{{\DD}-1})$ and so has expectation $1/n^{{\DD}-1}$
and variance $1/n^{2{\DD}-2}$. 
The summands are independent, and
the Chebyshev inequality can be used to show that the sum is concentrated around the mean.

For the upper bound we use a recent result of Johansson, Kahn and Vu \cite{JKV} 
on perfect matchings in random ${\DD}$-uniform hypergraphs. 
Their result implies that \whp\
there is a solution that only uses ${\DD}$-tuples of weight at most 
$\frac{K\log n}{n^{{\DD}-1}}$,
thereby giving an upper bound for ``Shamir's problem''.
While \cite{JKV} does not deal with ${\DD}$-partite hypergraphs,
we have verified that their result can be extended to this case;
this was also done independently by Gerke and McDowell for different purposes
\cite{GM}.
The upper bound for ${\DD}$-dimensional Axial assignment follows immediately. 
\end{proof}

The proof in \cite{JKV} is emphatically non-constructive.

\subsection{Main result} \label{mainthm}
We remind the reader that a description of the algorithm \BDTSx{\dd} 
for 3-dimensional Axial assignment ($\DD=3$)
will follow in Section \ref{PV2} (for $\dd=2$) 
and Section \ref{MG} (for general $\dd$).
Throughout the paper, 
we associate the element $\Mijk$ with the ``triple'' $(i,j,k)$.
\begin{theorem}\label{th2}
Suppose that $1\leq \dd\leq \e\log_2\log n$ where
where $0<\e<1/2$ is a constant. 
For a random 3-dimensional Axial assignment instance, \whp:
\begin{description}
\item[(a)] Algorithm \BDTSx{\dd} runs in time $O(n^3)$. 
\item[(b)] The cost $M(T)$ of the set of triples $T$ output by \BDTSx{\dd} satisfies
$ M(T) = O(2^{4\dd} n^{-1+\tk_\dd}\log n) $, 
where
$\tk_\dd =\frac{1}{2^{\dd+1}-1}$.
\end{description}
\end{theorem}
By taking $\dd=\e\log_2\log n$ we see that have a linear-time 
algorithm that \whp\ gives a solution of value $O(n^{-1+o(1)})$.
(The exponent's $o(1)$ is $\tfrac12 \log^{-\e}n$ plus lower-order terms.)
As already noted, we are not aware of any other polynomial-time algorithm that will \whp\
find a solution of value $O(n^{-1+o(1)})$,
or indeed any non-trivial bound including the $O(n^{-6/7})$ 
that \BDTSx{2} gives.

\section{3-Dimensional Axial assignment 2-level algorithm, \BDTSx{2}} \label{PV2}
In this section we consider a two-level version of the algorithm \BDTS\
for 3-dimensional Axial assignment.
In this way we hope to make it easier
to understand the general version described in Section \ref{MG}.
With reference to Theorem~\ref{th2},
the two-level version means taking $\dd=2$, $\tk=\tk_2=1/7$.
The algorithm has three phases,  
treated respectively in the next three subsections.

\subsection{Greedy Phase}\label{G1}
This phase uses a simple greedy procedure, 
shown below as Algorithm \ref{greedy2},
to construct a low-cost partial assignment of cardinality $(1-o(1))n$.
Let 
\begin{align} \label{params2}
n_1 &= n-n^{1-\tk} 
\qquad \text{and} \qquad 
w_0 = 2 n^{-2(1-{\tk})}\log n .
\end{align}
In the following we will pretend that $n_1$ is integral, 
as we will do throughout when it makes no difference.
There will be cases where integrality is significant,
and there we will be precise.

\begin{algorithm}[H]
\caption{\GreedyPhase: construct a cardinality-$n_1$ partial assignment}
\label{greedy2}
\begin{algorithmic}[1]
\STATE 
Let $B:=C:=[n]$, and $T:=\emptyset$ 
 (respectively the set of unassigned 2- and 3-coordinates,
 and the set of triples assigned so far)%
\FOR{$i=1,\ldots,n_1$}
\STATE Locate $\Mijk =\min\set{M_{i,j',k'} \colon j'\in B,k'\in C}$
\FAILIF {$\Mijk > w_0$}
\STATE Add $(i,j,k)$ to $T$ and remove $j$ from $B$ and $k$ from $C$
\ENDFOR
\STATE Return the set of triples in $T$ as a partial assignment
\end{algorithmic}
\end{algorithm}

\begin{lemma}\label{lem1}
With high probability the procedure does not fail, and yields 
$$Z_1 : = \sum_{(i,j,k)\in T} \Mijk 
\leq \frac{2}{n^{1-\tk}} . $$
\end{lemma}

\begin{proof}
We observe that if $(i,j,k)\in T$ then 
$\Mijk$ is the minimum of $(n-i+1)^2$ independent
copies of $\EX$ and is therefore distributed as $\expdist((n-i+1)^2)$.
Furthermore, the random variables $\Mijk, (i,j,k)\in T$ are independent
(each drawn from a different 1-plane as given by $i$).

The procedure fails only if some $\Mijk > w_0$.
Even in the last round $i=n_1$, the chance that all $(n^{1-\tk})^2$ elements 
considered are larger than $w_0$ is 
$\exp(-n^{2(1-\tk)} w_0) = \exp(-2 \log n)$.
By the union bound, the probability of a failure in any round
is $\leq n \exp(-2 \log n) = n^{-1}$.
We now ignore this aspect and imagine running the algorithm
without the failure condition.

Using the fact that an $\expdist(\l)$ random variable
has mean $1/\l$ and variance $1/\l^2$,
$$\E(Z_1)=\sum_{i=1}^{n_1}\frac{1}{(n-i+1)^2}\leq \int_{x=1}^{n_1+1}\frac{dx}{(n-x+1)^2}\leq \frac{1}{n^{1-\tk}}.$$
Now
$$\V(Z_1)=\sum_{i=1}^{n_1}\frac{1}{(n-i+1)^4}\leq \frac{3}{n^{3(1-\tk)}}=o(\E(Z_1)^2)$$
and the lemma follows from the Chebyshev inequality.
\end{proof}

\begin{remark}  \label{GreedyNote}
If we replace $n_1$ by $n-\om$ where $\om=\om(n)\to\infty$ and then do exhaustive search for the optimum 
solution in the remaining $\om$-size problem then we see that (i) the expected optimum value is bounded by 
$O\brac{\frac{1}{\om}+\frac{\log\om}{\om}}$ (greedy plus exhaustive search
costs), and (ii) the running time is bounded by $O\brac{n^3+\om^{2\om}}$. So, if $\om=O(\log n/\log\log n)$, the algorithm is polynomial
and produces a solution with a cost that is $o(1)$, in expectation and \whp.
\end{remark}
This justifies our introductory remark that it is
not difficult to prove $\ZAx{d}\to 0$ \whp\ for $\Mijk=\EX$.
(We do not even need the breakthrough result of \cite{JKV} for this. 
The early reult of Schmidt-Pruzan and Shamir \cite{SPS} would give $O\brac{\frac{1}{\om}+\frac{\log\om}{\om^{1/2}}}$
for the cost).

\subsection{Main Phase}\label{MP0}
This phase will increase the size of the partial assignment
defined by $T$ to $n-\fourm$.
The phase will proceed in \emph{rounds}.
At the start of each round we assume that 
the assignment elements $T_t$ are
$\triple1, \ldots, \triple{n_t}$,
which we can arrange
simply by permuting the array $M$.
Let $A=A(T)$ be the set of 1-coordinates assigned in $T$
and define $B$ and $C$ analogously for the 2- and 3-coordinates.
Of course initially $A=B=C=\set{1,\ldots,n_t}$ but the notation will be convenient.
Also, let $\bA(T) = [n]-A(T)$, the set of unmatched 1-indices,
and likewise define $\bB$ and $\bC$.

Round $t$ will reduce the number of unmatched elements
$\card{\bA}$ by a factor $\b$,
from $x_{t}$ to $x_{t+1}$,
while increasing the total cost of the matching by an acceptably small amount,
terminating in the round $\tau$ where $\card \bA$ is reduced to \fourm.
Because the sizes $x$ we are dealing with here become small, 
we cannot decently ignore integrality even though its effects are minor.
Specifically, then, we take
\begin{gather}
\a = 1/65
 , \qquad
\b = 1-\a,
\notag \\
x_1=n-n_1 = \floor{n^{1-\tk}} = \floor{n^{6/7}}
 \text{, and } x_t= \floor{ \b^{t-1}x_1} \text{ for $t\geq 2$}. \label{xt}
\end{gather}
Note that \eqref{xt} means that $x_t$ decreases by 1 
in each of the last few rounds, so the last round begins with $x_\tau=\four$.
It is clear that
\begin{align}
x_t \leq \b^{t-1}x_1
 ,\quad \text{thus } \quad
\tau \leq \log_{1/\b}(x_1 / \four)=O(\log n) .    \label{xtsize}
\end{align}

Recall from the Greedy Phase that 
\begin{align*}
 w_0 &=  n^{-12/7}\log n ,
\end{align*}
and, for $t \geq 1$, let
\begin{align}
\left.
\begin{aligned}
\pp_t &= n^{-6/7}x_t^{-8/7}\log n , 
\\
w_t &= -\log(1-\pp_t) = (1+o(1)) \pp_t,
\\
W_t &= w_0+w_1+\cdots+w_t .
\end{aligned}
\quad \right\} 
 \label{wdef}
\end{align}
At the start of round $t$ we will have revealed all elements (and only
elements) with values $\leq W_{t-1}$.
This is true in round $t=1$ by 
the Greedy Phase's definition. 
In later rounds $t$ we use the following ``good'' elements,
comprising a set $G_t$:
\begin{enumerate}
\item
Elements in the partial assignment $T_t$ at the start of the round.
\item
Elements with $\CC \in (W_{t-1}, W_t]$.
Conditioned upon $\CC>W_{t-1}$ and all history, 
a non-assignment element falls into this category with probability exactly 
$\pp_t = 1-\exp(-w_t)$.
\item
Previously revealed non-assignment elements, 
each taken with probability $\pp_t$,   
the selection made independently over all such elements.
\end{enumerate}
(In contrast to this ``sprinkling'' approach, 
\cite{CS99} used a slightly more complicated ``refreshing'' approach
to obtain small improvements to the constants.)
This selection assures the following property.
\begin{property} \label{goodset}
$G_t$ comprises only elements of weight $\leq W_t$, 
contains all elements in the partial assignment, 
and (conditioned upon all history)
contains each non-assignment element, independently, with probability 
$\pp_t$.
\end{property}

At the start of round $t$ let $A_t=A(T_t)$, $\bA_t=\bA(T_t)$,
and likewise for $B$ and $C$.
In round $t$ we will add $[n-x_t+1,n-x_{t+1}]$ to $A_t$. 
To add $i$ to $A_t$ we 
replace 6 of the triples in $T_t$ by 7 new triples
(see Figure~\ref{figfig}):
\begin{multline}\label{add}
+(i,j,k)-(j,j,j)-(k,k,k)+(j,p,q)+(k,r,s)-(p,p,p)-(q,q,q)-(r,r,r)-(s,s,s)+\\
(p,\x_1,\x_2)+(q,\x_3,\x_4)+(r,\x_5,\x_6)+(s,\x_7,\x_8).
\end{multline}
We call a collection of triples as in \eqref{add} an \emph{augmentation}.
Here $+$ indicates addition of a new triple to the assignment, 
$-$ indicates deletion of an existing assignment element,
$j,k,p,q,r,s \in A_t$ are all distinct, 
$\x_1,\ldots,\x_7 \in \bB_t$ are all distinct,
$\x_2,\ldots,\x_8 \in \bC_t$ are all distinct,
and each of the triples added in \eqref{add}
belongs to $G_t$ (thus has cost at most $W_t$).
To explain, we assign a new 1-coordinate $i$
to a \emph{previously used} 2-coordinate $j$ and 3-coordinate $k$
(unused ones are too rare and therefore too costly);
these uses of $j$ and $k$ conflict with the previous assignment elements 
$\triple j$ and $\triple k$ so we remove those elements from the assignment;
1-coordinates $j$ and $k$ are re-added as $(j,p,q)$ and $(k,r,s)$
thus colliding with the previous assignment elements
$(p,p,p)$, $(q,q,q)$, $(r,r,r)$, and $(s,s,s)$;
and finally 1-coordinates $p$, $q$, $r$, $s$
are re-added as $(p,\x_1,\x_2)$ etc.,
where the $\x_i$ are 2- and 3-coordinates \emph{not} previously assigned.
One may think of \eqref{add} as a binary tree version of
an alternating-path construction;
we will control the cost despite the tree's expansion.

We realise that the algorithm is not yet completely specified but we postpone questions
on index selection for a while longer. We first confirm that if we can always find $j,k,\ldots,\xi_8$ as in \eqref{add}
then the cost of the assignment produced by the Main Phase is acceptable.
Each application of
\eqref{add} increases the cost by $\leq 7 W_t$.
Success in a round means doing this $x_t-x_{t+1}$ times, in which
case the additional cost of the Main Phase will be at most 7 times
\begin{align}
\sum_{t=1}^{\tau}(x_t-x_{t+1})W_t
 &= x_1w_0 - x_{\tau+1} W_{\tau} +\sum_{t=1}^{\tau}x_tw_t ,
\notag 
\intertext{which (recalling \eqref{xtsize}) is at most $7+o(1)$ times 
}
 n^{-6/7}\log n+
 \log n
 &
 \sum_{t=1}^{\tau} x_t (n^{-6/7} x_t^{-8/7})
\notag \\
 & \leq n^{-6/7}\log n+
n^{-6/7} \log n \, \sum_{t=0}^{\tau-1} (\b^t x_1)^{-1/7} .
\notag
\intertext{The last is an increasing geometric series,
summing to 
$\leq ((1/\b)^\t/x_1)^{1/7} / (1-\b^{1/7})$.
By \eqref{xtsize}, $(1/\b)^\t \leq x_1/4$, so the expression above is}
\notag 
 & \leq n^{-6/7}\log n+
 n^{-6/7} \log n \: \frac{(1/\four)^{1/7}}{1-\b^{1/7}}
\notag \\
 & =O(n^{-6/7}\log n).
 \label{mainsum2}
\end{align}

With reference to Figure \ref{figfig2},
we now complete the description of the algorithm by explaining how 
we search for 
$j,k,\ldots,\xi_8$ that satisfy \eqref{add}. 
We show later that \whp\ all of our searches succeed.
Since we are considering just round $t$, 
we drop the superscripts and write, for example, $B$ rather than $B_t$.

We will first show how to 
{construct sets} $J=\SS 1 1,\ldots,\Xi_8=\SS 3 8$,
then show how to draw from these sets $\a x$ collections of elements
$j \in J$, \dots, $\xi_8 \in \Xi_8$,
each such collection augmenting the assignment as per \eqref{add}.
Within each augmenting set,
each chosen element will force the choice of the two elements below it.
For convenience, especially in the general case but even for 
the depth $k=2$ case now,
it is convenient to think of these sets as $\SS 1 1,\ldots,\SS 3 8$
as shown in the figure.

\begin{figure}[htb] 
\centering
  \begin{tikzpicture}[scale=0.8,thick]
   \tikzstyle{vertex}=
     [rounded corners,fill=blue!10,inner sep=1ex,outer sep=0.5ex]
   \tikzstyle{edge}=[-] 
   \tikzstyle{antiedge}=[decorate,decoration=zigzag,->,>=stealth',]

    \node [vertex] (c1) at (-5.5,-9) {$J = \SS 1 1$};
    \node [vertex] (c2) at (+4.5,-9) {$K = \SS 1 2$};

    \node [vertex] (e1) at (-8,-12) {$P = \SS 2 1$};
    \node [vertex] (e2) at (-3,-12) {$Q = \SS 2 2$};
    \node [vertex] (e3) at (+2,-12) {$R = \SS 2 3$};
    \node [vertex] (e4) at (+7,-12) {$S = \SS 2 4$};
    \draw [edge] (c1)--(e1);
    \draw [edge] (c1)--(e2);
    \draw [edge] (c2)--(e3);
    \draw [edge] (c2)--(e4);

    \node [vertex] (E1) at (-9.3,-15) {$\Xi_1 = \SS 3 1$};
    \node [vertex] (E2) at (-6.7,-15) {$\Xi_2 = \SS 3 2$};
    \node [vertex] (E3) at (-4.3,-15) {$\Xi_3 = \SS 3 3$};
    \node [vertex] (E4) at (-1.7,-15) {$\Xi_4 = \SS 3 4$};
    \node [vertex] (E5) at (+0.7,-15) {$\Xi_5 = \SS 3 5$};
    \node [vertex] (E6) at (+3.3,-15) {$\Xi_6 = \SS 3 6$};
    \node [vertex] (E7) at (+5.7,-15) {$\Xi_7 = \SS 3 7$};
    \node [vertex] (E8) at (+8.3,-15) {$\Xi_8 = \SS 3 8$};
    \draw [edge] (e1)--(E1);
    \draw [edge] (e1)--(E2);
    \draw [edge] (e2)--(E3);
    \draw [edge] (e2)--(E4);
    \draw [edge] (e3)--(E5);
    \draw [edge] (e3)--(E6);
    \draw [edge] (e4)--(E7);
    \draw [edge] (e4)--(E8);
   \end{tikzpicture}
\caption{\label{figfig2}}
\end{figure}

We construct the sets $\SS \ell m$ from bottom up.
At the bottom level, $\ell=3$,
``equipartition'' $\bB$ into sets $\Xi_1,\Xi_3,\Xi_5,\Xi_7$.
To be precise, partition $\bB$ into sets $\Xi_1,\ldots,\Xi_7$
each of size $\floor{ \card{\bB}/4 }$, and a fifth ``discard'' set.
(Normally we would ignore integrality, 
but for instance in the second-last round
$\bB$ will have size \fourp.)
Similarly, equipartition $\bC$ into $\Xi_2,\Xi_4,\Xi_6,\Xi_8$.
Note that for $x \geq 4$,
$x-4\floor{x/4} < 4 \leq 4 \floor{x/4}$, thus $x < 8\floor{x/4}$,
and thus each set $\Xi$ has cardinality at least $x/8$.

For $\ell<3$,
each set $U = \SS {\ell} m$ in Figure \ref{figfig2} 
is constructed from its two child sets 
$V = \SS\ellp{2m-1}$ and $W = \SS\ellp{2m}$
and from a set $X$ of ``available'' 1-coordinates
(soon to be defined precisely)
by invoking a procedure 
$\Triples(X, V, W)$.
Recall that $G=G_t$ is the set of ``good'' elements for this round.
Intuitively, we want $\Triples(X,V,W)$
to be something like
the set of all elements $x \in X$ for which 
there are a $v \in V$ and $w \in W$ such that $(x,v,w)$ is of low cost, i.e.,
something like
$\proj_1\left( (X \times V \times W) \cap G \right) , $
where $\proj_i$ denotes the projection of a set onto its $i$th coordinate.
The more complicated definition of \Triples below
enables Property \ref{tripleprop}
and simplifies the analysis later.

\begin{algorithm}[H]
\caption{\Triples: construct set $U$ from sets $X,V,W$ of available
 1-, 2-, and 3-indices}
\label{triples}
\begin{algorithmic}[1]
\STATE Input $V,W,X$.
\FAILIF {$\card X < n- n^\nt$} 
   \label{tripf1} 
 \STATE Partition $X = X' \cup X''$ where $\card{X'} = \xcard := \xcardl$
  and $\card{X''} \geq n^\nt$
  (think of $X''$ as a ``reserve'' set
  from which elements are shifted into $X'$ to preserve $\card{X'}$)  \label{reserve}
\STATE Let $\Ustar := \emptyset$
\FOR{each pair of elements $v \in V$, $w \in W$ in turn}
  \FOR{each element $x \in X$, in turn} \label{triples from}
   \IF{ $(x,v,w) \in G$}
    \STATE{ $\Ustar := \Ustar \cup \set{(x,v,w)}$ }
    \FAILIF{$X'' = \emptyset$} \label{tripf2}
    \STATE{ Where $x''$ is the first element of $X''$,
            let $X' := (X' \setminus x) \cup \set{x''}$,
	    and let $X'' := X'' \setminus x''$
	   } \label{killx}
   \ENDIF
  \ENDFOR \label{triples to}
\ENDFOR
\STATE Let $U := \proj_1 (\Ustar)$ 
\STATE Return $(\Ustar,U,X)$
\end{algorithmic}
\end{algorithm}
Claim \ref{tripf} will show that failures within \Triples are highly unlikely.

\begin{property} \label{tripleprop}
\Triples has the properties that
(1) each element $u \in U$ appears in exactly one triple $(u,v,w) \in \Ustar$
and thus is linked to a unique $v \in V$ and $w \in W$,
and 
(2) each pair $v,w$ occurs in $\Bin(\xcard,\pp)$ triples in $\Ustar$
and these binomials are independent for all $v$ and $w$.
\end{property}

\begin{proof}
The first property is immediate.
The second comes from observing that in the algorithm's 
Lines \ref{triples from}--\ref{triples to}
the size of $X'$ is kept constant at $n'$, while
Property \ref{goodset} 
means that each of the $n'$ elements considered is, independently,
present in $G$ with probability $\pp$. 
\end{proof}

Given \Triples, it is straightforward to define
all the sets in Figure \ref{figfig2}.
This is done by \SetConstruct, shown as Algorithm \ref{SetConstruct} below.

\begin{algorithm}[H]
\caption{\SetConstruct: for Main Phase round $t$, construct all sets $\SS \ell m$}
\label{SetConstruct}
\begin{algorithmic}[1]
\STATE Input: sets $A$, $B$, and $C$ of assigned 1-, 2-, and 3-indices. 
  (Elsewhere, notationally, we have assumed $A=B=C=[n-x_t]$.)
\STATE Let $X := {A}$ and $x:=n-\card X$ 
\STATE ``Equipartition'' $\bB$ into $\Xi_1\cup\cdots\cup\Xi_7$ 
  $= \SS 3 1 \cup \cdots \cup \SS 3 7$
  each of size $\floor{x/4}$, 
  discarding the $x-4 \floor{x/r}$ remaining elements of $\bB$
\STATE Likewise, equipartition
  $\bC$ into $\Xi_2\cup\cdots\cup\Xi_8$
  $= \SS 3 2 \cup \cdots \cup \SS 3 8$
\STATE Remembering that each application of \Triples makes $X$ smaller
for the next one, let
\begin{align*}
\begin{array}{rclclcl}
\outtrip P &=& \OTR{2}{1}&:=&\Triples(X, \SS 3 1,\SS 3 2)&=&\Triples(X, \Xi_1, \Xi_2)\\
\outtrip Q &=& \OTR{2}{2}&:=&\Triples(X, \SS 3 3,\SS 3 4)&=&\Triples(X, \Xi_3, \Xi_4)\\
\outtrip R &=& \OTR{2}{3}&:=&\Triples(X, \SS 3 5,\SS 3 6)&=&\Triples(X, \Xi_5, \Xi_6)\\
\outtrip S &=& \OTR{2}{4}&:=&\Triples(X, \SS 3 7,\SS 3 8)&=&\Triples(X, \Xi_7, \Xi_8)\\
\\
\outtrip J &=& \OTR{1}{1}&:=&\Triples(X,\SS 2 1,\SS 2 2)&=&\Triples(X,P,Q)\\
\outtrip K &=& \OTR{1}{2}&:=&\Triples(X,\SS 2 3,\SS 2 4)&=&\Triples(X, R,S)
\end{array}
\end{align*}
\STATE Return all sets $\SS \ell m$ and $\SSstar \ell m$ defined above
\end{algorithmic}
\end{algorithm}

Given a partial assignment $T$
with $i-1<n$ elements, 
augmenting $T$ to include 1-index $i$ 
as in \eqref{add}
and Figure \ref{treepic}
is easy as if
we can find an element $(i,j,k) \in G$ with $j \in J$, $k \in K$.
For, if this is so, 
by construction of $\Jstar$
there is a unique element $(j,p,q) \in \Jstar$
with $p \in P$ and $q \in Q$.
Then, by construction of $\Pstar$ there is a unique element
$(p,\xi_1,\xi_2) \in \Pstar$ with $\xi_1 \in \Xi_1$ and $\xi_2 \in \Xi_2$.
Likewise, $q$ leads uniquely to $\xi_3,\xi_4$, 
and $k$ leads uniquely to $r,s$ and these to $\xi_5,\ldots,\xi_8$.
This collection, corresponding to the triples added in \eqref{add},
immediately determines those subtracted: $-\triple j$, $-\triple k$, etc.,
and thus the full augmentation.
This procedure for producing an augmentation is spelled out 
as \MakeTree$(i,j,k)$, Algorithm \ref{MakeTree}.
Identifying $\xi_1,\ldots,\xi_8$ with $\ss 1 1,\ldots,\ss 1 8$,
this is written in a way that can be quoted for the general case.

\begin{algorithm}[H]
\caption{\MakeTree$(i,j,k)$: produce an augmentation $\outlist$}
\label{MakeTree}
\begin{algorithmic}[1]
\STATE Input: a collection of sets as in Figure \ref{figfig2}
 as output by Algorithm \ref{SetConstruct} (\MakeTree),
 and a triple $(i,j,k)$ with
 $i \in \bA$, 
 $j = \ss 1 1 \in \SS 1 1 = J$,
 $k = \ss 1 2 \in \SS 1 2 = K$.
\STATE Let $\outlist := { +(i,\ss 1 1,\ss 1 2) }$
\FOR{$\ell =$ 1 to 2}  \label{mt3}
 \FOR{$m = 1$ to $2^\ell$}
   \STATE (Note that $\ss \ell m \in \SS \ell m$ is already defined) \label{mt5}
   \STATE Let $\outnode$ be the unique triple 
          in $\SSstar \ell m$ with 1-coordinate $\ss \ell m$
	  (thus defining the values for Line \ref{mt5} 
	  when $\ell$ is incremented in Line \ref{mt3})
	  \label{mt6}
   \STATE Let $\outlist := \outlist, -\triple{\ss \ell m}, +\outnode$, i.e.,
     add these two triples to the output, one negated, one positive
 \ENDFOR
\ENDFOR
\STATE Return $\outlist$ 
\end{algorithmic}
\end{algorithm}

\begin{property} \label{setprop}
In any augmentation constructed by \MakeTree,
the first indices of all positive triples are distinct,
as are all the second indices, and all the third indices.
\end{property}

\begin{proof}
All first indices 
in the whole collection of sets $\SS \ell m$ 
(and therefore within one augmentation)
are distinct
because they are drawn from a set $X$ ($X = A$ initially),
and each first index produced is immediately removed from $X$.

Any second index $v$ comes from a triple in some set
$\SSstar \ell {m} \subset X \times \SS \ellp {2m-1} \times \SS \ellp {2m}$,
so that $v \in \SS \ellp {2m-1}$.
We treat this with two cases.
If $\ell<d$, then $v \in \SS \ellp {2m-1}$
is a first index in $\SSstar \ellp {2m-1}$.
Such values are all distinct from one another by distinctness of first
indices.
If on the other hand $\ell=d$, 
the previous case does not apply (there are no sets $\Sstar$ at depth $\ddp$),
but $v \in \SS \ddp {2m-1} = \Xi_{2m-1}$
and thus belongs to one of the parts of the equipartition of $\bB$.
Since in an augmentation each triple is drawn from a different set $\Xi$,
and the sets come from different parts of the equipartition of $\bB$,
all these values are distinct from one another.
Finally, no two values from different cases can be equal:
those in the first case belong to $A$ 
and those in the second case belong to $\bB$,
and these sets are disjoint since $A=B$
by our notational assumption that the partial assignment $T$
consists of triples of the form $\triple i$.
Third indices are treated by the same argument as second indices.
\end{proof}

We can then augment the assignment via \eqref{add} 
the output of \MakeTree. We now give a formal proof that \MakeTree allows a feasible augmentation.
Again, we write this in a way that can be quoted for the general case.

\begin{property}   \label{valid1}
The output of \MakeTree defines a valid assignment augmentation.
\end{property}

\begin{proof}
Working backwards through the elements of $\outlist$ we show inductively that 
each makes a valid change.
(\MakeTree works top-down; this proof works bottom-up.)
Consider adding a bottom-most triple
$+\sspar \dd m$.
Its second and third coordinates can be added, as they belong to the sets
$\SS\ddp{2m-1}$ and $\SS\ddp{2m}$ 
which contain only unassigned indices,
and (over all bottom-most triples)
there are no collisions amongst these
by Property \ref{setprop}.
The first coordinate would clash with the existing assignment element,
but this is removed by the $\outlist$ element $-\triple{\ss \dd m}$,
also freeing up $\ss \dd m$ as a 2- and 3-index.
So, adding the level $\ell=\dd$ elements,
pairing the positives and negatives (so the assignment size does not change),
ensures that all their first coordinates at level $\ell=\dd$ are available as 2- and 3-indices
at level $\ell-1$. We refer to this as property $\cQ$ at level $\ell=d$.

We now show that property $\cQ$ at level $\ellp$
implies property $\cQ$ at level $\ell$.
Consider a level-$\ell$ element $+\sspar \ell m$.
Referring to \MakeTree's Line \ref{mt6},
this element's 2-index $\ss {\ellp} {2m-1}$ (in iteration $\ell$)
is another element's 1-index (in iteration $\ellp$),
and thus, by Property $\cQ$ at level $\ellp$,
is available for assignment.
This 2-index does not conflict with any other,
by Property \ref{setprop}.
The same holds for the 3-indices.
The 1-index $\ss \ell m$ would conflict with an existing assignment element,
but this is taken care of by $-\triple{\ss \ell m}$.
That also frees up $\ss \ell m$ as an available 2- and 3-index,
completing property $\cQ$ at level $\ell$.

When $\ell=1$, $(i,\ss 1 1,\ss 1 2)$ can be added: 
the 2- and 3-indices are available by Property $\cQ$,
while 1-index $i$ is available 
(indeed the whole point of the procedure is to assign it).
\end{proof}

However, to make more than one such augmentation using the same 
data structure (the same output of \SetConstruct)
we must ensure, for example, that no two augmentations
use the same value $\xi_1$.
To this end, 
if a first augmentation uses ``leaf'' indices $\xi_1,\ldots,\xi_8$,
we regard the values $\xi_1, \ldots, \xi_8$ as ``poisoned'';
they in turn poison all ancestors depending on them,
including this augmentation's values $j,\ldots,s$, but also other values.
Contrapositively, any ``healthy'' (non-poisoned) elements 
$j' \in J$ and $k' \in K$'
depend only on healthy elements $\xi$. 

To be precise, a poisoned $\xi_1$ in turn poisons any ancestor
$p' \in P$ for which there exists a triple
$(p',\xi_1,\xi_2') \in \Pstar$.
Likewise the poisoned $\xi_2$ poisons any ancestor 
$p' \in P$ for which there exists a triple
$(p',\xi_1',\xi_2) \in \Pstar$.
Note that the poisoned $P$ values include the original $p$
but may include more values.
(Each $p$ is associated with a unique $\xi_1$ and $\xi_2$,
but the opposite is by no means true.)
We continue upward.
Any poisoned $p \in P$ poisons any $j' \in J$ 
for which there exists a triple $(j',p,q') \in \Jstar$,
and poisoned $q$ values poison additional elements of $J$.
The poisoned $J$ values certainly include $j$.
Likewise, the poisoned values $\xi_5,\ldots,\xi_8$ poison subsets of $R$ and $S$,
which in turn poison some of $K$.

In general, quite simply, any poisoned element poisons all its ancestors,
where ancestry is the transitive closure of the parent relation in which,
for any $(u,v,w) \in \SSstar \ell m$, $u$ is a parent of $v$ and $w$.
(Again, Property \ref{tripleprop} means that $u$ has only these children,
but $v$ and $w$ will typically each have additional parents.)
The step-by-step poisoning description given in the previous paragraph
will be useful for probabilistic analysis, 
and is shown as Algorithm \ref{poisonPropagate} in a form usable for
\BDTSx{\dd} generally.

\def\STATEx#1{{\def\alglinenumber##1{}\STATE #1}\addtocounter{ALG@line}{-1}}

\begin{algorithm}[H]
\caption{\PoisonPropagate: poison all ancestors of elements used in an
augmentation} % Poison propagation
\label{poisonPropagate}
\begin{algorithmic}[1]
\STATE Input: the output of \SetConstruct
  with some elements poisoned,
  and an augmentation. 
\STATE Poison the augmentation's leaf values $\ss{3}{1},\ldots,\ss{3}{8}$
  \label{poison1}
\FOR{$\ell =$ 2 to 1} \label{poison2}
 \FOR{$m = 1$ to $2^\ell$}
   \FOR{each $(u,v,w) \in \SSstar \ell m$}
    \IF{$v \in \SS \ellp{2m-1}$ or 
              $w \in \SS \ellp{2m}$ is poisoned}
       \STATE Poison 
         $u \in \SS \ell m$ 
    \ENDIF
   \ENDFOR
 \ENDFOR
\ENDFOR
\end{algorithmic}
\end{algorithm}

With this, we can present the complete Main Phase algorithm.

\begin{algorithm}[H]
\caption{\MainPhase: augment \GreedyPhase assignment to one of cardinality $n-(2^d-1)$} 
\label{mainalg2}
\begin{algorithmic}[1]
\FOR{rounds $t=1,\ldots,\tau$}
\STATE (Recall that the partial assignment has cardinality $n-x_t$)
\STATE Permute the cost matrix so that the partial assignment is
  \\
  $T_t = \set{ \triple 1,\ldots, \triple {n-x_t} }$
\STATE Let $A=B=C=[n-x_t]$ be the sets of assigned 1-, 2-, and 3-coordinates
\STATE Call \SetConstruct$(A,B,C)$ and
  consider all elements in all its sets to be healthy
  \label{MainSetConstruct}
\STATE (We will now assign the first $\ceil{ \a x_t}$ unassigned 1-indices 
        $i$, in turn)
\FOR{% the first $\ceil{ \a x_t} $ unassigned 1-indices $i$, namely for 
  $i=n-x_t+1,\ldots,n-x_{t+1}$}
\STATE Let $J'$ and $K'$ denote the healthy subsets of $J$ and $K$
\FAILIF{$I' := G_t \cap (\set{i} \times J' \times K') = \emptyset$}
\label{maintest}  
  \STATE
   Let $(i,j,k)$ be the lexicographically first element of $I'$ \label{line search}
  \STATE \label{line makeTree}
   Call \MakeTree$(i,j,k)$ to produce an augmentation $\F$ 
  \STATE
   Augment the assignment with $\F$
   \label{augmentation}
  \STATE Call \PoisonPropagate($\F$) 
\ENDFOR
\ENDFOR
\end{algorithmic}
\end{algorithm}
Claim \ref{mainf} will show that failures within \MainPhase are highly unlikely.
First we show that if it succeeds, it returns a valid partial assignment.

\begin{property} \label{validi}
Each augmentation in \MainPhase Line \ref{augmentation} is valid.
\end{property}

\begin{proof}
For the first augmentation in a round, this has already been 
justified by Property \ref{valid1}.
To justify a later augmentation, we will show that
none of its indices has appeared, in the same coordinate,
in any previous augmentation within the round.
Consider first a value $u$ that appears in a positive triple 
in an augmentation $\F$ and a previous augmentation $\F'$.
Bear in mind that the sets $\SS \ell m$ ($\ell=1,\ldots,\ddp$),
the output of \SetConstruct, are fixed for the round.

Consider first $u$ drawn from $\SS \ell m$ for augmentation $\F$,
with $\ell \leq \dd$.
By an argument like that used in the proof of Property \ref{setprop},
within $\F'$, $u$ must have been drawn from the same set $\SS \ell m$:
these sets are disjoint from one another, 
and disjoint from each set $\SS \ddp m$.
But this is a contradiction, since use of $\F'$ would have poisoned
$u \in \SS \ell m$, and $\F$ uses only healthy elements.

Next consider $u$ drawn from $\SS \ddp {2m-1}$ for $\F$,
making $u$ a 2-index in $\F$.
Again, by disjointness of the various sets,
as a 2-index in $\F'$, 
$u$ can only have been drawn from the same set $\SS \ddp {2m-1}$.
(The only other place $u$ can appear is in a set $\SS \ddp {2m'}$,
but that would make it a 3-index.)
This is a contradiction, since use of $\F'$ would have poisoned
$u \in \SS \ddp m$, and $\F$ uses only healthy elements.

The same argument goes for $u$ drawn from $\SS \ddp {2m}$ for $\F$.

Finally, consider the 1-index $i$ itself
whose assignment is the purpose of the augmentation $\F$.
It cannot appear as a 1-index in any earlier or later augmentation $\F'$:
it does not have the same unique ``root'' role
for any other augmentation $\F'$,
it cannot have been picked from any set $\SS \ell m$ with $\ell \leq \dd$ 
since these are subsets of $A$ while $i \notin A$,
and if picked from any set $\SS \ddp m$ it is a 2- or 3-index not a 1-index.

We now consider a conflict involving a negative triple $\triple u$ in $\F$,
implying that $u$ appears as a 1-coordinate (and not $i$) 
in a positive triple in $\F$.
Such a conflict with another augmentation $\F'$ 
cannot involve a negative triple
(since then $u$ would appear also as a 1-coordinate in $\F'$,
which we have already excluded),
so in $\F'$, $u$ must appear in a positive triple.
It cannot appear as a 1-index (already excluded).
The the only 2- and 3-indices in $\F'$ that are not also 1-indices
are those drawn from sets $\SS \ddp m$,
implying $u \in \bB = \bC$
and thus $u \notin A$,
contradicting the appearance of $u$ as a 1-index in $\F$.
\end{proof}

This concludes a proof of correctness of \MainPhase.
It remains to show that 
the various failure conditions are unlikely, notably that
the sets $J$ and $K$,
initially and when restricted to their healthy elements later,
are sufficiently large that failure in 
\MainPhase Line \ref{maintest} is unlikely.
To that end we first estimate the sizes of all the sets at the beginning of a round and
then analyze their losses due to poisoning.

We start by showing that the following ``nominal sizes''
are good estimates of the sizes of the sets $\SS \relax \relax$ at corresponding depths:

\begin{align}
\left.
\begin{aligned}
 \sz_3 &= \floor{x/4} = \Theta(x) \\ 
 \sz_2 &= \pp \xcard (\sz_3)^2 = \Theta( \pp n x^2) = \Theta( n^{1/7} x^{6/7}\log n) 
  \text{, and}
  \\
 \sz_1 &= \pp \xcard (\sz_2)^2 = \Theta( n^{3/7} x^{4/7}\log^2n) .
\end{aligned}
\quad \right\} 
  \label{nominalsizes}
\end{align}
Let
\beq{errdef}
\err=n^{-1/100}.
\eeq
We know that the sets $\SS 3 1,\ldots,\SS 3 8$
have size precisely $\sz_3$, and
we will show that \qs%
\footnote{A sequence of events 
$\cE_n,n\geq 0$ is said to occur \emph{quite surely} (\qs)
if $\Pr(\cE_n)=1-O(n^{-C})$ for every positive constant $C$.}\
the sets $\SS 2 1,\ldots,\SS 2 4$
all have cardinality $(1 \pm \err) \sz_2$,
and $\SS 1 1$ and $\SS 1 2$ 
have cardinality $(1 \pm \err)^2 \sz_1$,
where the notation
$A=(1\pm\err)^r B$ is shorthand for $(1-\err)^r B \leq A \leq (1+\err)^r B$.
This is part of the following more detailed claim,
for which we need one definition.

\begin{definition}
\label{subtriples}
For sets $X, S', S''$, and $\outtrip S = \Triples(X, S', S'')$,
let $\Sstar(\cdot,s',\cdot) \subseteq \Sstar$ 
be those elements generated by $s' \in S'$ 
(that is, having $s'$ as their second coordinate).
Symmetrically, 
let $\Sstar(\cdot,\cdot,s'') \subseteq \Sstar$ 
be those elements generated by $s'' \in S''$.
\end{definition}

\begin{claim}\label{cl2}
For each depth $\ell \in \set{1,\ldots,\kp}$ and each $m \in \set{1,\ldots,2^l}$:
\begin{enumerate}
\item
The nominal sizes in \eqref{nominalsizes}
satisfy $\sz_1 > \sz_2 > \sz_3 = \Theta(x)$.
\item
With $\d$ as in \eqref{errdef}, \qs,
$\card{\SS \ell m} = \ratr^{2^{\ddp-\ell}-1} \sz_\ell$.
\item
For every pair of sets $S' = \SS \ell {2m-1}$, $S''= \SS \ell {2m}$ 
(e.g., $P,Q$),
with $\Sstar = \SSstar{\ell-1}m$
(e.g., $\Jstar$),
\qs\ every $s' \in S'$ gives
$\card{\Sstar(\cdot,s',\cdot)} = \ratr \pp \xcard \card{S''}$
and symmetrically
every $s'' \in S''$ gives
$\card{\Sstar(\cdot,\cdot,s'')} = \ratr \pp \xcard \card{S'}$.
\end{enumerate}
\end{claim}

\begin{proof}
The first assertion is immediate from $n\gg x$ and thus $\pp n' x \gg 1$.%
\footnote{The notation $f(n) \gg g(n)$ is equivalent to $f(n) = \om(g(n))$,
 i.e., $f(n)/g(n) \to \infty$.}
We prove the rest of the claim by induction on $\ell$, 
starting with $\ell=3$ and working backwards to $\ell=1$.
For $\ell=3$, the second statement is immediate.

For any $\ell$, the first two statements imply the third.
To see this, 
Property \ref{tripleprop} implies that given $S''$, distributionally, 
\beq{asd1}
\card{\Sstar(\cdot,s',\cdot)} = \sum_{S''} \Bin(n',\pp)= \ \Bin(\card{S''} \xcard, \pp).
\eeq
To this we apply the Chernoff inequality
\beq{Chernoff}
\Pr( \abs{ \Bin(n,p)-np} \geq \err np) 
  \leq 2e^{-\err^2np/3}
   \qquad \text{for }  0 \leq \err \leq 1
\eeq
(an easy consequence of Alon and Spencer \cite[Theorem A.1.15]{AS}),
substituting $\card {S''} n'$ for $n$, $\pp$ for $p$, and $\d$ for $\d$.
Recalling from \eqref{wdef} that $\pp = n^{-6/7}x^{-8/7}\log n$,
and from \Triples (Algorithm \ref{triples}) that
$n'=n-2n^\nt$,
the expectation of the binomial in \eqref{asd1} is
\begin{align}
\card{S''} \xcard  \pp =\Omega(\s_\ell n\pp)
    = \Omega(x n\pp )
    = \Omega( n^{1/7} x^{-1/7}\log n) 
    = \Omega( n^{1/49}\log n ) .
    \notag 
\end{align}
 From this and \eqref{errdef} it follows that 
$\err^2 \card{S''} n' \pp = \omega(\log n)$,
and thus the exponent in the Chernoff inequality \eqref{Chernoff}
is of order $-\omega(\log n)$.
This assures that for a given $s'$, \qs,
\beq{aq1}
\card{\Sstar(\cdot,s',\cdot)}=(1\pm\err)\card{S''}\r n'.
\eeq
Likewise, for a given $s''$, 
$\card{\Sstar(\cdot,\cdot,s')}=(1\pm\err)\card{S'}\r n'$.
By definition of \qs\ (failure probability smaller than any polynomial),
the union bound shows that \qs\ 
\eqref{aq1} holds for \emph{every} $s'$,
yielding the third part of the claim.

Finally, truth of the second and third parts of the claim for $\ell$
implies that of the second part for $\ell-1$.
 From \eqref{aq1}
it follows that 
\begin{align}
\card{\SS {\ell-1} m}
 = \card{\Sstar} 
  &=\sum_{s'\in S'}\card{\Sstar(\cdot,s',\cdot)}   \nonumber 
 \\ & = \ratr \pp n' \card{S'} \card{S''}    \label{aaq2}
 \\
 &= \ratr \pp n' \card{\SS \ell {2m-1}} \card{\SS \ell {2m}}\nonumber
 \\
 &= \ratr \pp n' ((1 \pm \r)^{2^{\dd-\ell}-1} \sz_\ell)^2
  \quad \text{ by the inductive hypothesis}\nonumber
 \\
 &= \ratr^{2^{\dd-(\ell-1)}-1} \sz_{\ell-1} .\nonumber
\end{align}
\end{proof}

\begin{claim} \label{mainf}
There is \qs\ no failure in \MainPhase 
(Algorithm \ref{mainalg2} Line \ref{maintest}).
\end{claim}
\begin{proof}
Consider how the sizes of the sets $\SS \ell m$ 
are affected by poisoning during a round.
Claim \ref{cl2} part (3) shows that each element $s' \in S'$
has almost exactly equal representation in the parent set $\Sstar$.
Specifically (see also \eqref{aq1} and \eqref{aaq2}), \qs,
every ratio
$\linefrac{\card{ \Sstar(\cdot,s',\cdot)}}{\card \Sstar}$ 
satisfies
$$\frac{\card{ \Sstar(\cdot,s',\cdot)}}{\card \Sstar}
 =\frac{\ratr \card{S''}\r n'}{\ratr \pp n' \card{S'} \card{S''}}
=\frac{\ratr ^2}{\card{S'} }.$$
Thus each poisoned $s'\in S'$ poisons 
$\card{ \Sstar(\cdot,s',\cdot)}
 = \ratr^2{\card \Sstar} / \card{S'}$
triples associated with the parent $\Sstar$.
If $\l \card{S'}$ elements are poisoned (a $\l$ fraction of $S'$),
they poison $\leq \l \ratp^2 \card{\Sstar}$ elements of $\Sstar$
(a $\l$ fraction of $\Sstar$, up to the small error factor);
the inequality allows for possible repeated poisonings of the same elements.

After $\l \sz_3$ iterations within a round, 
a $\l$ fraction of $\SS 3 1 = \Xi_1$ is poisoned.
It therefore poisons a $\leq \l \ratp^2$ fraction of $\SS 2 1$.
These poisoned elements 
in turn poison a $\leq \l \ratp^4$ fraction of $\SS 1 1$.
The same holds for $\SS 3 3, \SS 3 5, \SS 3 7$, 
so, summing, we see that a $\leq 4 \l \ratp^4$ fraction of $\SS 1 1$ is poisoned.
Similarly, a
$\leq 4 \l \ratp^4$ fraction of $\SS 1 2$ is poisoned.

As long as these fractions are both at most $1/2$,
at least half the elements of $J$ are healthy, likewise for $K$.
In this case the number of good triples $(i,j,k)$,
with $i$ given by the round and iteration, 
and $j \in J$ and $k \in K$ both healthy,
dominates $\Bin( \frac14 \s_1^2, \pp)$,
whose expectation is 
$$\frac14 \s_1^2 \; \pp 
 = \Theta\left( (n^{3/7}x^{4/7}\log^2n)^2 \; (n^{-6/7}x^{-8/7} \log n) \right)
 = \Theta(\log^5n) . $$
Thus \qs\ there is some good triple $(i,j,k)$.

The poisoned fractions of $J$ and $K$, 
each $\leq 4 \l \ratp^4$, are $\leq 1/2$ if
$\l \leq  (1 - \err)^4/8$.
Earlier we parametrized the total number of iterations as $\a x$,
so $\l \sz_3 \leq \a x$,
for $\l \leq \a x/\sz_3 < 8 \a$
(the equipartitioning gives $\sz_3 > x/8$).
Taking $\a = (1/64) (1-\err)^4$, we see that the round can \qs\ 
continue even in the iteration after $\a x$.
We thus perform $\ceil{\a x}$ iterations within the round,
all succeeding \qs %.
\end{proof}

\begin{claim} \label{tripf}
There is \qs\ no failure in \Triples 
(Algorithm \ref{triples}), 
in Lines \ref{tripf1} or \ref{tripf2}.
\end{claim}

\begin{proof}
Recall that \Triples is called by \SetConstruct (Algorithm \ref{SetConstruct}),
in turn called by \MainPhase Line \ref{MainSetConstruct}.
 From Claim \ref{cl2} and \eqref{nominalsizes},
the total size of all sets $\SS  1 1, \ldots, \SS 3 8$
is dominated by the sizes of $J = \SS 1 1$ and $K = \SS 1 2$,
and \qs\ is $O(\s_1) = O(n^{3/7} x^{4/7} \log^2 n)$.
Recalling that $x=x_t \leq x_1 = n^{6/7}$,
this is $O(n^{46/49})$.

Within \SetConstruct, at the start of round $t$ we have
$\card X = \card A = n-x_t$;
this is smallest for round 1, where $\card A = n-x_1 \geq n-n^{6/7}$.
$X$ is depleted precisely of the elements 
that \Triples adds to the sets $\SS \ell m$,
so by the paragraph above, even at the end of \SetConstruct, \qs\ 
$\card X 
  \geq n-n^{6/7}-O(n^{46/49})
  > n-n^\nt
$.
Thus Line \ref{tripf1} of \Triples \qs\ never fails.
Also, Line \ref{tripf2} fails iff the size of $U$ reaches $n^\nt$,
but again $U$ is one of the sets $\SS \ell m$,
and the first paragraph shows that its size is \qs\ smaller than this.

A union bound ensures that \qs\ there is no failure
within the polynomial-time execution of \MainPhase.
\end{proof}

\subsection{Final Phase}\label{FP}
We now have to add only \fourm indices to $I$.
At this point, when the number of unused indices is
$\card \bB = \card \bC = x<\four$, 
the previous approach does not work:
to begin with, there are not even enough unused indices
to define nonempty sets $\Xi_1,\ldots,\Xi_8$.
In the extreme case $x=1$, only $\xx$ is available to fill the
roles of $\x_1,\ldots,\x_7$ and $\x_2,\ldots,\x_8$,
but (within each of the two groups) we relied on all these values being distinct
to obtain a valid assignment
(see Figure \ref{treepic}).
We now show how to modify our approach in order to deal with 
this problem. 

Because there are only $O(1)$ indices to add, 
we can now afford to have a round consist of the addition of a single triple. 
It is easy to check that $W_\t = O(n^{-6/7} \log n)$,
so for the remaining rounds $t$
(namely $t=\t+1$, $\t+2$, and $\t+3$),
in lieu of \eqref{wdef} we may as well let 
\begin{align} 
\pp_t &= \pp =n^{-6/7} \log n , 
\label{rhofinal2}
\end{align}
defining $w_t$ and $W_t$ as in \eqref{wdef}.
Then, if successful, the cost of this phase is
$O(W_{\tau}+w_\tau)=O(n^{-6/7}\log n)$.

As usual, by permuting the cost matrix we can assume that
$T = \set{\triple i \colon i \in A}$, so $A=B=C$.
For notational convenience we further assume that $A=[n-1]$.
We can do so because, when adding one more element to $T$,
we can confine ourselves (in the algorithm and the analysis)
to using indices at most $\card A+1$;
when $x=3$, for example, this is like ``pretending'' that $n$ is $n-2$
and $x$ is $1$.

Consider the following relabeling of Figure \ref{treepic}. This indicates our strategy for completing
the assignment.
It is easy to check that if $n,j,\ldots,s$ are distinct then this
is a valid augmentation,
so the task is to find such an augmentation of low cost.

\begin{figure}[H] 
\centering
  \begin{tikzpicture}[scale=0.6,thick]
   \tikzstyle{vertex}=
     [rounded corners,fill=blue!10,inner sep=1ex,outer sep=0.5ex]
   \tikzstyle{edge}=[->,>=stealth']
   \tikzstyle{antiedge}=[decorate,decoration=zigzag,->,>=stealth',]

    \node [vertex] (a1) at (0,0) {$+(n,j,k)$};
   
    \node [vertex] (b1) at (-4,-3) {$-(j,j,j)$};
    \node [vertex] (b2) at (+4,-3) {$-(k,k,k)$};
    \draw [antiedge] (a1)--(b1);
    \draw [antiedge] (a1)--(b2);

    \node [vertex] (c1) at (-4,-6) {$+(j,p,q)$};
    \node [vertex] (c2) at (+4,-6) {$+(k,r,s)$};
    \draw [edge] (b1)--(c1);
    \draw [edge] (b2)--(c2);

    \node [vertex] (d1) at (-6,-9) {$-(p,p,p)$};
    \node [vertex] (d2) at (-2,-9) {$-(q,q,q)$};
    \node [vertex] (d3) at (+2,-9) {$-(r,r,r)$};
    \node [vertex] (d4) at (+6,-9) {$-(s,s,s)$};
    \draw [antiedge] (c1)--(d1);
    \draw [antiedge] (c1)--(d2);
    \draw [antiedge] (c2)--(d3);
    \draw [antiedge] (c2)--(d4);

    \node [vertex] (e1) at (-6,-12) {$+(p,n,p)$};
    \node [vertex] (e2) at (-2,-12) {$+(q,q,j)$};
    \node [vertex] (e3) at (+2,-12) {$+(r,k,r)$};
    \node [vertex] (e4) at (+6,-12) {$+(s,s,n)$};
    \draw [edge] (d1)--(e1);
    \draw [edge] (d2)--(e2);
    \draw [edge] (d3)--(e3);
    \draw [edge] (d4)--(e4);

   \end{tikzpicture}
  \caption{\label{treepica}}
\end{figure}
It will be convenient for this and the corresponding general case to equipartition $[n-1]$ into sets $N_{\ell,m}$ of size 
$n_0=n/(2^{\ddp}-2)$ for $1\leq \ell\leq \dd$ and $1\leq m\leq 2^\ell$
(6 sets in this $\dd=2$ case).
Drawing our indices (in this case $j,\ldots,s$)
from corresponding sets ensures distinctness.
Here all sizes are large and we will ignore integrality.
\begin{algorithm}[H]
\caption{\FinalPhase: augment the assignment by one}
\label{FinalPhase2}
\begin{algorithmic}[1]
\STATE Let $P = \set{p \in N_{2,1} \colon (p,n,p) \in G_t}$
\FOR{$j\in N_{1,1}$}
\STATE Let $Q(j)=\set{q \in N_{2,2} \colon (q,q,j) \in G_t}$
\ENDFOR
\STATE Let $J=\set{j\in N_{1,1}\colon (\exists p\in P,q\in Q(j)) \text{ with }(j,p,q)\in G_t}$  \label{jgood}

\vspace{0.2cm} 
\STATE Let $S= \set{s \in N_{2,4}\colon (s,s,n) \in G_t}$
\FOR{$k\in N_{1,2}$}
\STATE Let $R(k)=\set{r \in N_{2,3} \colon (r,k,r) \in G_t}$
\ENDFOR
\STATE Let $K =\set{k\in N_{1,2}\colon (\exists r\in R(k),s\in S) \text{ with }(k,r,s)\in G_t}$      \label{kgood}

\vspace{0.2cm} 
\FAILIF{$G_t \cap (\set{n} \times J \times K)) = \emptyset$}  \label{final2fail}
\STATE  \label{final2valid}
  Let $(n,j,k)$ be an element in this set.
  Let $p,q$ be values satisfying Line \ref{jgood} for $j$,
  and let $r,s$ be values satisfying Line \ref{kgood} for $k$.
\STATE
 Augment the assignment (as in Figure \ref{treepica}), adding 
 $+(n,j,k)$, $+(j,p,q)$, $+(k,r,s)$,
     $+(p,n,p)$, $+(q,q,j)$, $+(r,k,r)$, $+(s,s,n)$ 
 and removing 
      $-(j,j,j)$, \ldots, $-(s,s,s)$.
\end{algorithmic}
\end{algorithm}

\begin{claim} \label{final2f}
There is \qs\ no failure in \FinalPhase 
(Algorithm \ref{FinalPhase2})
in Line \ref{final2fail}.
\end{claim}

\begin{proof}
Observe first that $|P|$ is distributed as $\Bin(n_0,\pp)$. 
This has expectation $\Omega(n^{1/7}\log n)$ and so 
the Chernoff bounds imply that \qs\ $|P|=\ratr n_0\pp$, with $\err$ as in \eqref{errdef}.
Given $P$ and $j\in N_{1,1}$, 
the size of $Q(j)$ is distributed as $\Bin(n_0,\pp)$ 
and so \qs\ $|Q(j)|= \ratr n_0\pp$. 
The
trials for choosing a $Q(j)$ are independent of those for choosing $P$
as they use distinct first indices
(drawn respectively from $N_{2,2}$ and $N_{2,1}$).
The trials for the various $Q(j)$ are all independent
as they use distinct last indices.
Condition on the sets $P$
and $Q(j)$ for $j\in N_{1,1}$ and that they are as large as \qs\ claimed.

The trials to determine if $j \in N_{1,1}$ belongs to $J$ 
are the only trials to look at matrix elements with first index $j$,
so they are independent of the trials for $P$, for every $Q(j)$,
and for every other $j'$.
Fix $j\in N_{1,1}$ and let $\a_j$ be the (conditional) probability that $j\in J$. Then
$$\a_j
 =1-(1-\pp)^{|P|\,|Q(j)|}
 \geq \frac23 |P|\,|Q(j)|\pp
 \geq\frac12n_0^2\pp^3.$$
Because the trials for each $j \in N_{1,1}$ are independent,
$|J|$ dominates $\Bin(n_0,n_0^2\pp^3/2)$ and \qs\
this is at least $n_0^3\pp^3/3$. 
Conditioning on all the foregoing, 
the same analysis shows that \qs\ $|K|\geq n_0^3\pp^3/3$. Given this, 
and recalling $\pp$ from \eqref{rhofinal2},
the probability that no $(n,j,k)$ is good is
at most 
\begin{align*}
(1-\pp)^{n_0^6\pp^6/9}
 & \leq \exp(-\pp n_0^6\pp^6/9)
 = \exp(-\Theta( n^6 \pp^7 ))
 = \exp(-\Theta( \log^7 n))
.
\end{align*}
\end{proof}

This completes the analysis of \BDTS\ when there are two levels, \BDTSx 2.

\section{3-Dimensional Axial assignment general algorithm, \BDTSx d} 
\label{MG}
We follow the same three-phase strategy as for \BDTSx 2
but with depth
$\dd$ a positive integer satisfying
\begin{align}
2 \leq \dd\leq \e\log_2\log n
 \quad \text{ where } \quad
0<\e<1/2 .
\label{grandparams}
\end{align}
For the Greedy and Main Phases
the ideas are the same as before;
only the calculations are more difficult.
For the Final Phase, new ideas are needed to generalize the
construction illustrated in Figure \ref{treepica}.

\subsection{Greedy Phase}  \label{greedy general}
This is much as before.
Proceed as in Section \ref{G1}
but take 
\begin{gather*}
\tk = \tk_\dd=\frac{1}{2^\ddp-1} 
 , \qquad
n_1 = n-n^{1-\tk}
 , \qquad
w_0 = n^{-2(1-{\tk})}\log n .
\end{gather*}
Lemma \ref{lem1} continues to hold.

\subsection{Main Phase: parameters and cost bound}\label{MP1}
In parallel with Section \ref{MP0}, let
\begin{align}
 \label{parameters}
\left.
\begin{gathered}
\a = \tfrac{1}{10} 2^{-2\dd} 
 , \qquad
\b = 1-\a,
\\ 
x_1 = \floor{ n^{1-\tk} } , \qquad
\quad
x_t 
 = \floor{ \b x_{t-1} } \text{ for $t \geq 2$,} 
\\ 
 \qquad
\pp_t = x_t^{-1-\tk}n^{\tk-1} \log n,\text{ for $t \geq 1$,}
\\ 
\tau = \argmax_t \left\{ x_t \geq 2^\dd \right\}
 ,\qquad 
 w_t = -\log(1-\pp_t), \quad
W_t = w_0+w_1+\cdots +w_t .
\end{gathered}
\quad \right\}
\end{align}
Note that
\begin{align*}
x_t \leq \b^{t-1}x_1
 ,\qquad 
\tau \leq \log_{1/\b}(x_1 / 2^\dd)=O(\log^2n)
 ,\quad \text{and } \quad
w_t = (1+o(1)) \pp_t .
\end{align*}

As before let $T_t$ be the partial matching (set of triples)
at the start of round $t$ of the Main Phase, 
let $A_t = \proj_1(T_t)$ be the set of 1-indices assigned
and $\bA_t$ those not assigned, similarly for $B$ and $C$,
and again for notational convenience
assume that $T_t = \triple 1,\ldots,\triple{n-x_t}$.
Round $t$ consists of $\ceil {\a x_t}$ iterations,
each increasing the cardinality of the partial assignment by 1,
so that the round reduces the number of unassigned elements 
from $x_t$ to $\floor{\b x_t} = x_{t+1}$.
As before,
the small size of $\a$ means that the last several rounds each
each run for a single iteration and thus the last phase will have
$x_\t = 2^\dd$ exactly.
In analogy with \eqref{add}, each augmentation
will add $2^\ddp-1$ triples each of cost at most $W_t$,
and remove $2^\ddp-2$ triples.

Altogether the Main Phase increases the partial assignment's cost by at most
\begin{align*} 
(2^\ddp-1)
\sum_{t=1}^{\tau}(x_t-x_{t+1})W_t
 & \leq 
   2^\dd   \parens{ x_1w_0+\sum_{t=1}^{\tau}x_tw_t }
 \\ & = 
   O(2^\dd) \parens{
  n^{{\tk}-1} \log n +
   n^{\tk-1} \log n \sum_{t=1}^{\tau} x_t^{-\tk} 
   } .
\\ & = O(2^{4\dd} n^{\tk-1} \log n) .
\end{align*}
The last line above follows 
in precise analogy to \eqref{mainsum2},
using
\begin{align}
\sum_{t=1}^{\t} (x_t)^{-\tk}
 \leq \sum_{t=0}^{\t-1} (\b^t x_1)^{-\tk}
 \leq \frac{ ((1/\b)^\t)/x_1)^\tk} {1-\b^\tk}
 \leq \frac{(2^{-\dd})^\tk} {1-\b^\tk}
 \leq O(2^{3d}),
 \label{mainsum}
\end{align}
whose last inequality comes from
$\dd \, \tk = \Theta(1)$
and $\b^{\tk} = (1-\a)^\tk = 1 - \Omega(\a \tk) = 1-\Omega(2^{-3\dd})$.

\subsection{Main Phase: algorithms and properties}\label{MP2}

The algorithm follows lines that we hope are clear from the $\dd=2$ case,
with only small generalizations needed.

\mainalgsec{Property \ref{goodset}} holds just as before.

\mainalgsec{Algorithm \ref{triples}}, \Triples,
constructing a set $U$ of triples from 
two child sets $V,W$ of triples and a set $X$ of available 1-indices,
is as before with the sole exception that we replace
$n^\nt$ with $n^{1-\tk/10}$

\mainalgsec{Property \ref{tripleprop}} holds for the same reason as before.

\mainalgsec{Algorithm \ref{SetConstruct}}, \SetConstruct, 
is a trivial extension of its $\dd=2$ special case.
It is given as Algorithm \ref{SetConstruct} below,
and constructs the obvious generalization of Figure \ref{figfig2}.

\begin{algorithm}[H]
\caption{\SetConstruct: for Main Phase round $t$, construct all sets $\SS \ell m$}
\label{mainalg1g}
\begin{algorithmic}[1]
\STATE Input: sets $A$, $B$, and $C$ of assigned 1-, 2-, and 3-indices. 
\STATE Let $X = A$ and $x=n-\card X$.
\STATE ``Equipartition'' $\bB$ into 
 $\SS \ddp 1 \cup\SS \ddp 3 \cup\cdots\cup \SS \ddp {2^\ddp-1}$ 
  each of size $\floor{x/2^\dd}$, 
  discarding the $x-2^\dd\floor{x/2^\dd}$ remaining elements of $\bB$
\STATE Likewise, equipartition
  $\bC$ into 
 $\SS \ddp 2 \cup\SS \ddp 4 \cup\cdots\cup \SS \ddp {2^\ddp}$ 
\FOR {$\ell = \dd,\ldots,1$}
  \FOR {$m = 1, \ldots, 2^\ell$}
     \STATE Let $\outtripp \ell m = 
          \Triples(X, \SS \ellp {2m-1} , \SS \ellp {2m}) $
  \ENDFOR
\ENDFOR
\STATE Return all sets $\SS \ell m$ and $\SSstar \ell m$ defined above
\end{algorithmic}
\end{algorithm}

\mainalgsec{Algorithm \ref{MakeTree}}, \MakeTree, 
is as shown earlier except that $\ell$ ranges from 1 to $\dd$,
and of course the names $J$ and $K$ are to be disregarded.

We will use an augmentation output by \MakeTree
just as before, adding all positive triples to the assignment
and removing all negative triples from it.
Instead of adding 7 triples and removing 6, 
as in \eqref{add} for the $\dd=2$ special case,
we now add $2^\ddp-1$ triples and subtract $2^\ddp-2$.

\mainalgsec{Properties \ref{setprop} and \ref{valid1}} 
hold for the same reasons as before.

\mainalgsec{Algorithm \ref{poisonPropagate}}, \PoisonPropagate,
again has no changes 
except that in Line \ref{poison1} 
the poisoned leaves are $\ss \ddp 1, \ldots, \ss \ddp {2^\ddp}$,
and that Line \ref{poison2} iterates through $\ell=\dd$ to 1.

\mainalgsec{Algorithm \ref{mainalg2}}, \MainPhase 
is unchanged but for
interpreting $J$ as $\SS 1 1$ and $K$ as $\SS 1 2$,
and likewise their starred variants.

\mainalgsec{Property \ref{validi}} holds for the same reason as before.

This establishes that \MainPhase is correct, 
and we need only show that \qs\ it does not fail. 
To this end we will prove generalizations of 
Claims \ref{cl2}, \ref{mainf}, and \ref{tripf},
starting with Claim \ref{cl2}'s determination of the sizes of the
sets $\SS \ell m$.

Letting 
$n'=n-2n^{1-\tk/10}$,
in analogy with \eqref{nominalsizes} we define
nominal sizes of level-$\ell$ sets by
$\sigma_\ddp = \floor{x/2^\dd}$ 
and, for $\dd \geq \ell \geq 1$,
\begin{align}
\sigma_\ell &= \pp n' \sigma_\ellp^2 ,  \label{recurrence}
\intertext{a recurrence leading easily to}
 \sigma_{\ell} &= (\pp n')^{2^\kpell-1} (\sigma_\ddp)^{2^\kpell} 
 = \frac1{\pp n'} (\pp n' \sigma_\ddp)^{2^\kpell}  
 . \label{recurrencesolved}
\end{align}

As in the depth-2 case, we will show that the sizes of each set $\SS \ell m$
is approximated by the nominal size $\sz_\ell$ up to a factor
$\ratr^{2^{d+1-\ell}}$,
with 
\beq{deltag}
\d = n^{-\tk^2/3}. 
\eeq
We note that these bounds are good, by assertion 0 of the following claim.

\begin{claim}  \label{sizeclaim}
For each depth $\ell \in \set{1,\ldots,\kp}$ and each $m \in \set{1,\ldots,2^l}$:
\begin{enumerate}
\setcounter{enumi}{-1}
\item
For $\d$ as defined in \eqref{deltag},
$\ratr^{2^{\dd}} = 1+o(1) .$%

\item
The nominal sizes in \eqref{recurrencesolved} satisfy 
$\sz_1 \gg \cdots \gg \sz_\kp = \Theta(x/2^\dd)$.

\item
\Qs,
$\card{\SS \ell m} = \ratr^{2^{\ddp-\ell}-1} \sz_\ell$.

\item
For every pair of sets $S' = \SS \ell {2m-1}$, $S''= \SS \ell {2m}$,
with $\Sstar = \SSstar{\ell-1}m$,
\qs\ every $s' \in S'$ gives
$\card{\Sstar(\cdot,s',\cdot)} = \ratr \pp \xcard \card{S''}$
and symmetrically
every $s'' \in S''$ gives
$\card{\Sstar(\cdot,\cdot,s'')} = \ratr \pp \xcard \card{S'}$.
\end{enumerate}
\end{claim}

\begin{proof}
The proof mirrors that of Claim \ref{cl2} but with more computation.
We start with some observations we will use repeatedly.
By definition,
$$\tk=\Theta(2^{-d})=\Omega(\log^{-e}n) . $$
Then
\begin{align} \label{neta}
 n^{\tk^2} 
   = \exp(\log n \, \Omega(\log^{-2\e}n) ) 
   = \exp(\Omega(\log^{1-2\e}n)) .
\end{align}
(This is the reason we require $\e<1/2$.)

We first prove part 0 (which was not worth stating for the $d=2$ case).
First,
\begin{align*}
 (1+\d)^{2^\kp} 
   & \leq \exp(\d 2^\kp) 
   = \exp( n^{-\tk^2/3} \Theta(1/\tk) ) .
\end{align*}
By \eqref{neta}, the argument of the exponential satisfies
\begin{align}
  n^{-\tk^2/3} \Theta(1/\tk) 
   &= \exp(\Omega(\log^{1-2\e}n)) O(\log^\e n)
   = o(1),  \label{netadom}
\end{align}
where the last step uses $\e<1/2$ and follows from the general observation
that the exponential of a logarithm to any positive power
dominates any polylogarithmic quantity.
(Taking logarithms, $\log$ to any positive power dominates $\log \log$.)
Thus,
$$ (1+\d)^{2^\kp} = \exp(o(1)) = 1+o(1) , $$
establishing assertion 0.

The first assertion's last equality follows from $x \geq 2^\dd$
(so that $\floor{x/2^\dd}$ is not 0).
Proving the rest of the first assertion inductively on $\ell=\dd,\ldots,1$,
suppose that $\sz_\ellp \geq \sz_\ddp$; the base case $\ell=d$ is trivial.
Then, recalling that $x \leq n^{1-\tk}$,
we see that
\begin{multline}  \label{rhon}
  \pp n' \sz_\ellp
   \geq \pp n' \sz_\ddp
   = \Omega(\pp n x \, 2^{-\dd})
   = \Omega( (n/x)^\tk \log n \, \log^{-\e}n)
   = \Omega( n^{\tk^2} \log^{1/2} n)
   \\
   = \exp(\Omega(\log^{1-2\e}n)) \log^{1/2} n 
   = \omega(1),
\end{multline}
where the last line makes use of \eqref{neta} and the previous observation
on domination.
 From \eqref{rhon} and \eqref{recurrence}, 
$
 \pp n' {\sz_\ellp}^2
  = (\pp n' \sz_\ellp) \sz_\ellp
  \gg \sz_\ellp
$,
establishing the first assertion.

We prove the rest of the claim by induction on $\ell$, 
starting with $\ell=\ddp$ and working backwards to $\ell=1$.
We equate 
$S' = \SS \ell {2m-1}$, $S''= \SS \ell {2m}$,
and $\Sstar = \SSstar{\ell-1}m$.
For $\ell=\ddp$, the second statement is immediate.
For any $\ell>1$, the first two statements imply the third
(which is empty when $\ell=1$).
To see this,
Property \ref{tripleprop} implies that given $S''$, distributionally, 
\beq{asd1b}
\card{\Sstar(\cdot,s',\cdot)} = \sum_{S''} \Bin(n',\pp)= \ \Bin(\card{S''} \xcard, \pp).
\eeq
We wish to apply the Chernoff bound to this.
 From assertion~0 and the inductive hypothesis,
$\card{S''} = (1+o(1)) \sz_\ellp = \Omega{\sz_\ddp} , $
and thus the Chernoff bound is exponentially small in 
\begin{align*}
 \d^2 \card{S''} n' \pp
  & \geq (n^{-\tk^2/3})^2 (1+o(1)) \Theta(2^{-\dd} x) n (x^{-1-\tk} n^{\tk-1})
  \\& = (n^{-2 \tk^2/3}) \Theta(\tk) (n/x)^{\tk}
\quad \text{ which, reasoning as in \eqref{rhon}, is}
  \\& = \Omega (n^{+\tk^2/3} \tk) 
  = \exp(\tfrac12 \Omega(\log^{1-2\e}n)) \, \Omega(\log^{-\e} n)
  = \omega(\log n)
.
\end{align*}
Thus the Chernoff bound
is smaller than any polynomial,
so the binomial's value is \qs\ almost exactly 
equal to its expectation.
Thus, \qs\ for all $s' \in S'$,
$
 \card{\Sstar(\cdot,s',\cdot)} = \rate \card{S''} n' \pp ,
$
with the symmetric statement holding for all $s'' \in S''$.
This gives the claim's third assertion for $\ell$.

Then, simply by summing over all $s' \in S'$,
$
 \card{\Sstar} = \rate \card{S'} \card{S''} n' \pp .
$
By the inductive hypothesis,
$$
 \card{\Sstar} = \rate \left[\rate^{2^{(\kp)-(\ell)}-1} \sz_\ell\right]^2 n' \pp
               = \rate^{2^{\kp-(\ell-1)}-1} \sz_{\ell-1} .
$$
This establishes the claim's second assertion for $\ell-1$,
concluding the induction and the proof.
\end{proof}

The next claim parallels Claim \ref{mainf}
and is proved by identical reasoning.

\begin{claim} \label{mainfg}
There is \qs\ no failure in \MainPhase.
\end{claim}

\begin{proof}
Claim \ref{sizeclaim}'s first assertion
establishes the sizes of the sets $\SS 1 1$ and $\SS 1 2$,
and we now consider the sizes of their healthy subsets.
After $\l \sz_\ddp$ iterations within a round, 
by definition a $\l$ fraction of $\SS \ddp 1$ is poisoned.
Reasoning as in the proof of Claim \ref{mainf},
Claim \ref{sizeclaim}'s second assertion
says that each element in the child set
is responsible for an equal number of elements of the parent,
to within a factor $\rate^2$,
and thus the poisoned $\l$ fraction of $\SS \ddp 1$
poisons a fraction $\rate^2 \l$ of $\SS \dd 1$.
This in turn poisons a fraction $\rate^4 \l$ of $\SS {\dd-1} 1$
and so on, 
finally poisoning a fraction $\rate^{2\dd}$ of $\SS 1 1$.
Each leaf set $\SS \ddp m$ poisons this same fraction of $\SS 1 1$,
so altogether, the fraction of $\SS 1 1$ poisoned is at most
$2^{\dd} (1+\d)^{2\dd} \l$.
(Half the leaves lead to $\SS 1 1$; the other half act on $\SS 1 2$.)

We wish to ensure that at most half of $\SS 1 1$ is poisoned.
This will be so if
$2^{\dd} (1+\d)^{2\dd} \l \leq 1/2$.
Claim \ref{sizeclaim}'s assertion~0 tells us that 
$(1+\d)^{2d} = 1+o(1)$,
so it suffices to ensure that $\l \leq \tfrac13 \, 2^{-\dd}$.
The round runs for $\ceil{\a x}$ iterations
and we are not interested in the state after the last iteration,
so $\l \sz_\ddp \leq \floor{\a x} \leq \a x$.
Thus it suffices to take 
$\a x/\sz_\ddp \leq \tfrac13 \, 2^{-\dd}$,
or equivalently
$ \a 
  \leq (\sz_\ddp/x) \, \tfrac13 \, 2^{-\dd} 
$.
The right-hand side of this inequality is equal to
$ (\floor{x/2^\dd}/x) \, \tfrac13 \, 2^{-\dd} 
  \geq (\tfrac12 2^{-\dd}) \, (\tfrac13 \, 2^{-\dd})
  = \tfrac1{6} \, 2^{-2\dd} 
$,
so it suffices to ensure that $\a \leq \tfrac16 2^{-2\dd}$;
this is done by its definition in \eqref{parameters}.

 From Claim \ref{sizeclaim}'s assertion 2 and then assertion 0,
$\card{\SS 1 1},\card{\SS 1 2} = \ratr^{2^\ddp} \sz_1 = (1+o(1)) \sz_1$,
so the product of the sizes of their healthy subsets is 
$\geq \frac15 {\sz_1}^2$, and
$\card{ G_t \cap \set{i} \times \SS 1 1 \times \SS 1 2 }
 \sgeq \Bin({{\sz_1}^2}/5, \pp)$, where the symbol $\sgeq$ denotes stochastic domination.
By Chernoff, this is \qs\ nonzero if the 
expected value, $\pp {\sz_1}^2 /5$, is $\om(\log n)$, as we now verify.

First, recalling that $2^d = \Theta(1/\tk)$, we note that 
$$(n'/n)^{2^\dd} 
  = (1-n^{-\tk/10})^{\Theta(1/\tk)}
  = 1- \Theta(  n^{-\tk/10} \cdot 1/\tk) )
  = 1-o(1), 
$$
where the last step uses 
$$
  n^{-\tk/10} \cdot 1/\tk
   = \exp(-\log^{1-\e}n) O(\log^\e n) = o(1) ,
$$
as comes from reasoning like that for \eqref{neta} and \eqref{netadom}.
By \eqref{recurrencesolved}
and using $\sz_{\dd+1} > \tfrac12 x/2^\dd$,
the expectation $\pp {\sz_1}^2 /5$
is at least one fifth of
\begin{align*}
\pp \frac{1}{(\pp n')^2} \, (\pp n' \s_\kp)^{2^\kp}
&\geq \frac{1}{n^2\pp}\bfrac{\pp n x}{2^\ddp}^{2^\kp}\\
&= \bfrac{x}{n}^{1+\tk}\frac{1}{\log n}\brac{\bfrac{n}{x}^\tk\frac{\log
n}{2^\ddp}}^{2^\kp} 
\\
&=\frac{1}{\log n}\bfrac{\log n}{2^\ddp}^{2^\ddp} 
\quad \text{(using $(2^\ddp) \tk = 1+\tk$)}
\\
&\geq \frac{1}{\log n}\bfrac{\log n}{\log^\e n}^{2^\ddp} 
\\
&\geq (\log n)^{2^\dd-1}.
\end{align*}
That is, the expectation is of order $\omega(\log n)$ for any $d>1$
(quick inspection shows it is also so for $d=1$),
whereupon the Chernoff bound shows that \qs\ every set
${ G_t \cap \set{i} \times \SS 1 1 \times \SS 1 2 }$
is nonempty, and there is no failure in \MainPhase.
\end{proof}

\newcommand{\ntd}{{1-\tk/10}}

The next claim parallels Claim \ref{tripf}
and is proved by identical reasoning.
\begin{claim}
There is \qs\ no failure in \Triples 
(Algorithm \ref{triples}), 
in Lines \ref{tripf1} or \ref{tripf2}.
\end{claim}

\begin{proof}
By analogy with the proof of the $\dd=2$ special case, 
it suffices to show that the sum of the sizes of all sets $\SS \ell m$
produced by \SetConstruct (calling \Triples)
is $o(n^\ntd)$,
so by Claim \ref{sizeclaim} it suffices to show that the sum 
of all the nominal sizes is this small.
 From Claim \ref{sizeclaim}'s assertion~1,
each set's nominal size dominates the sum of the sizes of its two children,
so the sum of all the sets' nominal sizes is less than twice $\sz_1$.
We now verify that this is $\ll n^{1-\tk/10}$.
We use that $2^\dd-1 < 2^\dd-1/2 = 1/(2\tk)$.
First,
$$
\s_1
  =\frac1{\pp n'} (\pp n' \s_{\dd+1})^{2^\dd}
  \leq \frac1{\pp n'} (\pp n x)^{2^\dd}
  =  x(\pp n' x)^{2^\dd-1}
    =x((n/x)^\tk \log n)^{1/(2\tk)}
    =(nx)^{1/2} (\log n)^{1/(2\tk)} .
$$
This is increasing with $x$, and $x \leq x_1 = n^{1-\tk}$, so
$$
\s_1 
 \leq 
 n^{1-\tk/2} (\log n)^{1/(2\tk)}
 = \exp( (1-\tk/2) \log n + {1/(2\tk)} \log \log n) .
$$
Recalling that $\tk=\Omega(\log^{-\e}n)$,
the term $(\tk/2) \log n$ is $\Omega(\log^{1-\e} n)$
while the final term is $O(\log^\e n \, \log \log n)$,
so we conclude that 
\begin{align} \label{sigma1}
\s_1 
 \leq \exp( (1-\tk/3) \log n)
 = n^{1-\tk/3} 
 = o( n^{1-\tk/10} )
 .
\end{align}
\end{proof}

\subsection{Final Phase}
We execute the Main Phase so long as
the number of unassigned elements is at least $2^\dd$.
In the Final Phase, then, we must add $2^\dd-1$ triples,
and as in Section \ref{FP}, 
we will proceed in rounds, 
each round adding just one triple.
In a round, there are $x < 2^\dd$ unassigned elements.
We assume
that the assigned first indices are $A=[n-x]$,
and that the partial assignment consists of the set of triples
$T=\{(i,i,i) \colon i\in A\}$.
We simplify the notation of Section \ref{MP1} and 
let
\begin{gather}
 \label{simpler} 
\pp=n^{\tk-1} \log n 
, \quad 
w = -\log(1-\pp) = \pp (1+o(1)),
\end{gather}
with $\tk = 1/(2^\ddp-1)$ as in Sections \ref{greedy general} and \ref{MP1}.

Because each round adds just one triple, 
if successful, the cost of this phase is 
$$O(2^{\dd}(W_{\tau}+{2^\dd} w))=O({2^{2\dd} n^{\tk-1} \log n}).$$
To see this, the $2^{2\dd}w$ on the left clearly matches the bound.
Recalling $W_\t=w_0 + \sum_{t=1}^{\t} w_t$,
the first term's contribution $2^\dd w_0 = 2^\dd n^{2(\tk-1)} \log n$
is negligible.
Then, 
$\sum_{t=1}^{\t} w_t = n^{\tk-1} \log n \: \sum_{t=1}^{\t} x_t^{-1-\tk}$
is summed as in \eqref{mainsum} to give
$ n^{\tk-1} \log n \: \frac{(2^{-d})^{1+\tk}}{1-\b^{1+\tk}} $.
Still as in \eqref{mainsum},
$(2^{-d})^{1+\tk} = \Theta(2^{-d})$ 
while
$\b^{1+\tk} = 1-\Omega(\a(1+\tk)) = 1-\Omega(\a)$.
So this term, with its $2^d$ prefactor, contributes
$ 2^d n^{\tk-1} \log n \: \frac{O(2^{-d})}{\Omega(2^{-2\dd})} $,
again matching the bound.

A certain function $\f$ will play a key role in defining a general
set of augmenting triples, analogous to the $\dd=2$ case
illustrated in Figure \ref{treepica}.
The function $\f$ may be thought of as taking as input a binary string
(interpretable as a binary number),
along with its length,
and dropping the final all-0 or all-1 block (of maximum length).
That is, we define
\begin{align*}
 \f(\ell, \bitsi \ell) &= (i, \bitsi i)
 \quad \text{where $b_{i} \neq b_{i+1}$ but $b_{i+1}=\cdots=b_\ell$} .
\end{align*}
If $b_1 = \cdots = b_\ell$ then $\f(\ell, \bitsi \ell)=(0,\emptystring)$
where \emptystring denotes the empty string.

Consider a complete binary tree $\cT$ 
on levels $0,\ldots,\dd$,
with nodes $(\ell,m)$ for $\ell=0,\ldots,\dd$ and $m=0,\ldots,2^\ell-1$,
where the children of $(\ell,m)$ are $(\ellp,2m)$ and $(\ellp,2m+1)$.%
\footnote{%
We take the range of $m$ to be $0,\ldots,2^\ell-1$ 
(``index origin'' 0 rather than 1)
because this is notationally better given 
our use of the binary representation of $m$.
}
Make the natural correspondence between a number and its binary
representation, with the empty string $\emptystring$ having value 0.
Then $\f$ establishes a map from the nodes at any level of $\cT$
(notably the leaves) to nodes at a higher level,
\begin{align*}
\f(\ell,m) &= (i,m') ,
\end{align*}
where we interpret $m$ as its $\ell$-bit binary representation $\bitsi \ell$,
and $m'$ is the number with $i$-bit binary representation $\bitsi i$.
For a given level $\ell$, this map from level-$\ell$ nodes to
shallower nodes is injective except for 
the two nodes $(\ell,0)$ and $(\ell,2^\ell-1)$ mapping to the root.
This follows because, except at the root, $\f$ is invertible, with
$\f^{-1}(i,\bitsi i)) 
 = (\ell, (b_1 \ldots b_i (1-b_i) \ldots (1-b_i))$,
where the final binary number has $\ell$ bits.

The mapping $\f$ will be used to define a generic set of triples that
can be used to augment an assignment,
as illustrated in Figure \ref{phi3} for $\dd=3$.
It is analogous to the triples illustrated in Figure \ref{treepica}
for $\dd=2$,
with the indices $1,\ldots,15$ here playing a role
similar to that of the variables $n,j,k,\ldots,s$ there.
Each node $u=(\ell,m)$ of $\cT$ is given a triple 
$$(i_u,j_u,k_u) . $$
The first index $i_u$ is given by
\beq{sigmalm}
i_u \equiv i_{\ell,m} := 2^{\ell}+m
, \quad \text{and for the root,} \quad i_{(0,\$)}=1 .
\eeq
(Any unique labeling will do, though we will capitalize on
the parities of the leaves.)
For a non-leaf $u$ with left child $v$ and right child $w$ 
we complete the triple as $(i_u,i_v,i_w)$.
A leaf $u=(\dd,m)$ with $m$ even
is given triple $(i_u,i_{\f(u)},i_u)$,
and one with $m$ odd
is given triple $(i_u,i_u,i_{\f(u)})$. 
(The role of $\f$ is critical, as we will explain shortly.)

\begin{figure}[H] 
\centering
\begin{tikzpicture}
[level distance=1.0cm,
level 1/.style={sibling distance=8 \mysep},
level 2/.style={sibling distance=4 \mysep},
level 3/.style={sibling distance=2 \mysep},
]
\node{ \textbf{\textit{1}} 2 {3}}
 child {node{2 4 5}
  child {node{4 8 9}
   child {node{8 \textbf{\textit{1}} 8}
   }
   child {node{9 9 4}
   }
  }
  child {node{\textbf{5} {10} {11}}
   child {node{{10} \textbf{5} {10}}
   }
   child {node{11 {11} 2}
   }
  }
 }
 child {node{\textbf{3} 6 7}
  child {node{6 {12} 13}
   child {node{{12} \textbf{3} 12}
   }
   child {node{13 13 6}
   }
  }
  child {node{7 14 15}
   child {node{14 7 14}
   }
   child {node{15 15 \textbf{\textit{1}}}
   }
  }
 }
 ;
\end{tikzpicture}
\caption{For $\dd=3$, construction, based on $\f$, 
of a generic tree of triples $(i,j,k)$ to augment the assignment.
  \label{phi3}}
\end{figure}
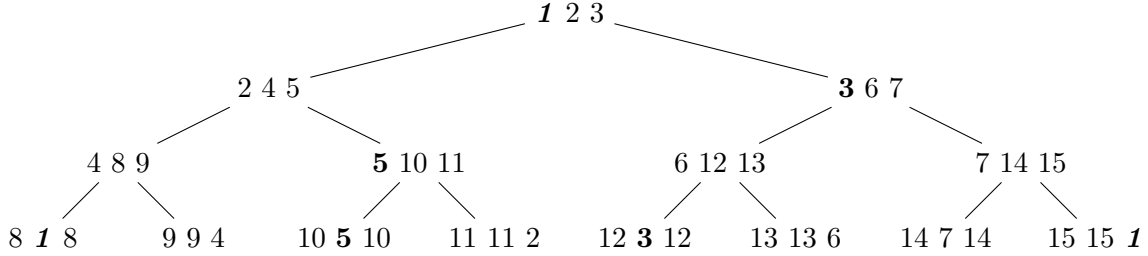

The figure shows in bold font linkages between various leaves and
internal nodes via $\f$.
The first and last leaves, with $m=0$ and $m=7$, 
have binary representations $000$ and $111$,
both mapping to the root, 
and thus their triples both include the root's value $i_u=1$.
The third leaf maps to the second level-2 node, 
since $\f(3,010)=(2,01)$, 
so its triple $(10,5,10)$ includes that node's value $i_u=5$.
Likewise, the fifth leaf maps to the second level-1 node since
$\f(3,100)=(1,1)$,
so its triple $(12,3,12)$ includes that node's value $i_u=3$.

As in the $\dd=2$ special case of the Final Phase (Section \ref{FP}),
to avoid dependencies we equipartition 
the assigned first indices
$[n-x]$ into sets $N_{\ell,m}$ of size 
\begin{align}
n_0 = \frac{n-x}{2^\ddp-2} ,  \label{simpler2}
\end{align}
for $1 \leq \ell \leq \dd$
and $0\leq m\leq 2^{\ell}-1$
(both $x$ and $2^\ddp$ are of order $o(n)$,
so we can and shall ignore integrality),
and we define a singleton set $N_{0,0}=\set{n}$:
we will, without loss of generality, discuss the addition of first index $n$ to the assignment in this round.

We instantiate the generic tree of triples via a mapping
\begin{align}
 \p \colon \set{1,\ldots,2^\ddp-1} \rightarrow \oneto {n} 
 \quad \text{with} \quad
\p(i_{\ell,m}) \in N_{\ell,m} .
 \label{pidef}
\end{align}
The function $\p$ plays the same role as the assignment of 
values to the variables $n,j,k,\ldots,s$ of Figure \ref{treepica},
just in a more general notation.
For example where before we assigned a value to $\ss 1 1=j$,
now we choose a value $\p(i_{1,0})=\p(2)$;
another important example is that
our choice of $N_{0,0}$ implies that $\p(i_{0,0})=\p(1)=n$.
To augment the assignment,
for every triple $(i,j,k)$ in the tree we add
$(\p(i), \p(j), \p(k))$ to the assignment,
and, for every $i$ in $2,3,\ldots,{2^\ddp-1}$,
we delete the old assignment triple $\triple{\p(i)}$.
We will seek a mapping $\p$ such that the added triples
are all of low cost.
We denote the augmentation corresponding to $\p$ by $A_\p$.

\begin{claim}\label{cl100}
The augmentation $A_\p$ described above yields a proper assignment.
\end{claim}
\begin{proof}
Let us ignore $\pi$ for the moment, imagining it to be the identity.
We first show that each of the values $r=1,\ldots,2^\dd-1$
occurs exactly three times in the triples, 
once each as a first, second, and third element.

First consider $r=1$.
This value appears as the first index $i_u=1$ of the root triple, 
the second index in the triple at leaf $(\dd,0)$,
and the third index in the triple at leaf $(\dd,2^\dd-1)$.

Next consider $2^\ell \leq r < 2^{\ellp}$ with $1 \leq \ell < \dd$.
Writing $r=2^\ell+m$,
$r$ occurs as the first index $i_u$ of a node $u=(\ell,m)$
which is neither the root nor a leaf.
(For example, $r=3$ in node $(1,1)$ in Figure \ref{phi3}.)
If $m$ is even, then $u$ is the left child of its parent $v=(\ell-1,m/2)$
and so $r$ appears as the second index $j_v$ of the triple at $v$.
Let $w=(\dd,m')$ be the unique leaf for which $\f(w)=u$. 
The definition of $\f$ implies that $m'$ is odd,
and thus $r = i_u = i_{\f(w)}$ appears 
as the third index $k_w$ of the triple at $w$. 
The case $m$ odd is handled similarly.

Finally, consider $2^\dd \leq r < 2^\ddp$.
Write $r=2^\dd+m$.
Assume first that $m$ is even. 
Then $r$ appears as the first and third indices
of the triple at $(\dd,m)$ and the second index of its parent $(d-1,m/2)$. 
(For example, $r=10$ in the leaf $(3,2)$ and its parent in Figure \ref{phi3}.)
If $m$ is odd then $r$ 
appears as the first and second indices
of the triple at $(\dd,m)$ and the third index of its parent $(d-1,(m-1)/2)$. 

By counting, these three occurrences of each value $r$
account for all elements of all triples, so none occurs 
anywhere other than in the positions described.

Now consider the instantiated triples, of the form
$(\p(i_u), \p(j_u), \p(k_u))$.
By definition \eqref{pidef}, $\p$ is injective,
so each value $\p(r)$ occurs exactly once in each coordinate,
including $\p(1)=n$.
Also by \eqref{pidef}, for each $r=2,\ldots,2^{\dd+1}-1$, 
$\triple{\p(r)}$ was previously in the assignment,
but $\triple{\p(1)}=\triple n$ was not.
Thus, after the augmentation $A_\p$ adds the 
first set of triples and deletes the second,
the assignment is augmented by $\triple n$.
\end{proof}

Recalling that $\p(1)=n$,
our aim now is to find a set of values $\p(2),\ldots,\p(2^{\dd+1}-1)$
respecting \eqref{pidef} and
making all the added triples cheap.
For the subtree $\cT_u$ rooted at $u$,
let $\p'$ be 
a $\p$-like map but on a domain restricted to indices occurring in triples
outside the subtree $\cT_u$.
(Critically, some of these indices may also occur in $\cT_u$.)
We say that an extension to $\cT_u$ 
(possibly a completion to all of $\cT$) $\p$ of $\p'$ is good
if, for every $v \in \cT_u$, $(\p(i_v),\p(j_v),\p(k_v)) \in G$.
(The definition does not consider the cost of triples outside $\cT_u$,
something determined entirely by $\p'$,
though of course we will seek to ensure that they too are of low cost.)
\begin{claim}
\label{goodSubtrees}
For any non-root vertex $u$,
the only indices occurring both in triples in $\cT_u$ and outside it
are $i_u$ (occurring twice inside $\cT_u$ and once outside) 
and $i_{\f(u)}$ (occurring once inside and twice outside). 
Given the cost matrix,
whether $\p'$ has a good extension to $\cT_u$ depends solely on the values of
$\piip u$ and $\piip {\f(u)}$.
\end{claim}

\begin{proof}
We prove the claim by induction on the depth of $u$, 
from the leaves to the root.
We first establish the base cases.
If $u$ is a leaf, 
$i_u$ occurs twice in $u$'s triple 
(thus once outside it),  
and 
$i_{\f(u)}$ occurs once in $u$'s triple (thus twice outside it),
and of course there are no other indices in $\cT_u$.
Since no index occurs exclusively in $\cT_u$, 
$\p=\p'$.
For even $u$, $\p$ is good on $\cT_u$ iff $\ptripp u {\f(u)} u \in G$,
a function of $\piip u$ and $\piip {\f(u)}$ alone;
for odd $u$ a similar argument applies.

Now consider $u$ with left child $v$ and right child $w$,
starting with the case that $u$ is even 
($u=(\ell,m)$ with $m$ even).
The triple at $u$ is $(i_u,j_u,k_u) = (i_u, i_v, i_w)$.
By definition of $\f$,
$\f(v)=\f(u)$, $\f(w)=u$.
Also, the four nodes ${\f(v)}={\f(u)}$,
${\f(w)}={u}$, $v$, and $w$ are all distinct
(using that $u$ is not the root)
as are their corresponding indices $i$.
What indices can occur both within and outside $\cT_u$?
The only candidates are the new index $i_u$
and, inductively from the two subtrees,
$i_v$, $i_{\f(v)}$, $i_w$, and $i_{\f(w)}$.
Inductively, $i_v$ makes only one appearance outside $\cT_v$,
and that is accounted for by its appearance as the second element of 
the triple at $u$.
The same goes for $i_w$, appearing as the third element of the triple at $u$.
Two appearances of $i_{\f(v)}$ lie outside $\cT_v$,
and since $\f(v)=\f(u)$ is an ancestor of $u$
(again using that $u$ is not the root), lying outside $\cT_u$,
there are (as desired) two appearances of $i_{\f(u)}$ outside $\cT_u$.
Finally, two appearances of $i_{\f(w)}$ lie outside $\cT_w$,
and since $\f(w)=u$, one of these appearances is 
the first element $i_u$ of the triple at $u$,
leaving (as desired) one appearance of $i_u$ outside $\cT_u$.
A similar argument applies to $u$ odd, with the roles of $v$ and $w$ swapped.

Since $i_u$ and $i_{\f(u)}$ are the only indices appearing
both outside $\cT_u$ (thus in the domain of $\p'$)
and inside it,
and since for any other index in $\cT_u$ we are free to choose 
the value of $\p$ as we like
(subject only to the prescription in \eqref{pidef} that $\p(u) \in N_u$)
it is immediate that whether $\p'$ has a good extension $\p$
depends only on the evaluation of $\p'$ at these two points.
\end{proof}

Taking advantage of Claim \ref{goodSubtrees},
we define a pair $(a,b)$ to be good for $\cT_u$ if,
with $\piip {\f(u)} = a$ and $\piip u = b$,
$\p'$ has a good extension $\p$ to $\cT_u$.

\begin{remark}  \label{remConstruct}
We can recursively, bottom-up, construct all good pairs 
for every non-root node of $\cT$.
For a leaf $u$, $(a,b)$ is good iff $(b,a,b) \in G$ (for even $u$)
or $(b,b,a) \in G$ (for odd $u$). 
For a non-root, non-leaf even $u$ with left child $v$ and right child $w$,
let $\piip {\f(u)} = \piip {\f(v)} = a \in N_{\f(u)}$ and
$\piip u = \piip {\f(w)} = b \in N_u$.
Then $(a,b)$ is good for $\cT_u$,
i.e., $\p'$ has a good extension $\p$ to $\cT_u$,
iff there exist values
$\pii v = c \in N_v$ and $\pii w = d \in N_w$
such that the triple at $u$ is good, 
i.e., $(b,c,d) \in G$,
and the subtrees at $v$ and $w$ have good extensions,
which by the inductive hypothesis is to say that
$(a,c)$ is good for $\cT_v$, and $(b,d)$ is good for $\cT_w$.
A similar calculation works for $u$ odd.
\end{remark}

For a node $u=(\ell,m)$, with $a \in N_{\f(u)}$, it proves convenient to define
\begin{align*}
\SS \ell m(a) &= \set{ b \in N_u \colon (a,b) \text{ is good for } \cT_u}
\end{align*}
Algorithm \ref{mainalgy} simply implements
the leaf-up iterative application of
the goodness test described in Remark \ref{remConstruct}.%
\footnote{This iterative construction can be viewed as a dynamic program.
Specifically, it is a dynamic program on
a tree decomposition of a hypergraph, of a sort described by
Scott and Sorkin \cite[Theorem 6 and its proof]{SS}.
We cannot exploit \cite{SS} because we need more details 
to do our probabilistic counting, but we will state the connection.
Consider a hypergraph whose vertices are the indices of $\cT$
and whose hyperedges are its triples.
The critical aspect of the connection 
(and of our function $\f$) is that $\cT$ gives a 
bounded-treewidth tree decomposition of the hypergraph:
the ``bag'' at a node $u$ consists of the three indices $i$, $j$, $k$ 
in its triple and (for internal nodes)
the (unique) index $\f(i)$ that occurs both below and above $u$.

Let $f(\p)$ be a function on labelings which is
obtained as the product over edges $e$ of functions $f_e$ of the 
labeling of edge $e$.
\cite{SS} shows how to efficiently solve the problem of finding 
a labeling maximizing $f$, 
and indeed related counting problems,
in the general setting of graphs of small treewidth;
an extension to hypergraphs is merely notational.
In this case we take the function for each hyperedge
to have value 1 if the labeled triple is good (low cost), 0 otherwise,
so that the product function $f$ of the whole labeling is
1 iff the labeling makes all triples good,
and the Scott-Sorkin algorithm will find a good labeling if one exists.
}
It generalizes Algorithm \ref{FinalPhase2} (\FinalPhase for $\dd=2$);
it may also be seen as playing a role analogous to Algorithm \ref{mainalg1g}
though it is entirely different in its details.

\begin{algorithm}[H]
\caption{Final Phase SetConstruct: construct sets used to augment
 the assignment by one}
\label{mainalgy}
\begin{algorithmic}[1]
\FOR{$m=1,\ldots,2^{\dd}-1$}
 \FOR{$a \in N_{\f({d,m})}$}    \label{aline1}
  \STATE Let       \label{Sline1}
    \begin{align*}
     \SS d m (a) &=
      \begin{cases}
       \set{b\in N_{d,m} \colon (b,a,b)\in G} & \text{for $m$ even}, 
       \\
       \set{b\in N_{d,m} \colon (b,b,a)\in G} & \text{for $m$ odd} .
      \end{cases}
     \end{align*}
 \ENDFOR
\ENDFOR
\FOR{$\ell=\dd-1,\ldots,1$}
 \FOR{$m=0,\ldots,2^\ell-1$}
 \FOR{$a \in N_{\f({\ell,m})}$}    \label{aline2}
   \STATE Let       \label{Sline2}
    \begin{align*}
     \SS \ell m (a) &=
     \begin{cases}
      \rule{0cm}{0.0cm}
      \set{b \in N_{\ell,m} \colon \exists 
             c\in \SS {\ellp}{2m}(a), \: d\in \SS {\ellp}{2m+1}(b),
             \text{ s.t.} (b,c,d) \in G} & \text{for $m$ even},
      \\
      \rule{0cm}{0.6cm}
      \set{b\in N_{\ell,m} \colon \exists 
             c\in \SS {\ellp}{2m}(b), \: d\in \SS {\ellp}{2m+1}(a),
	     \text{ s.t.} (b,c,d) \in G} & \text{for $m$ odd} .
     \end{cases}
    \end{align*}
 \ENDFOR
 \ENDFOR
\ENDFOR
\STATE 
Let $\SS 0 0 = \left( \set n \times \SS 1 0(n) \times \SS 1 1(n) \right)
          \cap G$.
\end{algorithmic}
\end{algorithm}

After Algorithm \ref{mainalgy} completes,
the round of Final Phase completes as follows.
If $\SS 0 0 = \emptyset$, \fail.  
Otherwise, choose any triple $(n,a,b) \in \SS 0 0$ 
and define a partial mapping by 
$\p'(0,0)=n$ (as always),
$\p'(1,0)=a$, and $\p'(1,1)=b$.
By construction, $(n,a,b) \in G$ is good,
$a \in \SS 1 0(n)$ means that $(n,a)$ is good for $\cT_{1,0}$
so $\p'$ has a good extension $\p''$ to this subtree, and 
$b \in \SS 1 1(n)$ means that $(n,b)$ is good for $\cT_{1,1}$
and thus $\p''$ has a completion $\p$ which is good for this subtree
and thus for all of $\cT$.
That is, if $\SS 0 0$ is nonempty,
a good set of labels for all nodes can be found by ``unwinding''
Algorithm \ref{mainalgy} back down to the leaves.
We will not bother to give the procedure explicitly.
This gives a valid augmentation, as guaranteed by Claim \ref{goodSubtrees}.

\subsection{Probability of success}
It remains now to show that Algorithm \ref{mainalgy} succeeds \whp,
producing a nonempty set $\SS 0 0$.
Let us first point out that $N_{0,0}$, the odd set with
cardinality 1 rather than $n_0$ (defined in \eqref{simpler2}),
only ever occurs in Algorithm \ref{mainalgy}'s Lines \ref{aline1} and \ref{aline2},
affecting the number of choices for $a$
and thus the number of sets $S_u(a)$ defined,
but never in Lines \ref{Sline1} and \ref{Sline2}
determining the sizes of these sets.

Let
\begin{gather} \label{simpler2y}
\s_\ell = (\pp n_0)^{2^{\dd-\ellp}-1}
\text{ for } 0 \leq \ell\leq \dd ,
\end{gather}
Note that
\begin{align}
\s_\ell &=\pp n_0 \s_{\ellp}^2, 
\quad
 \pp n_0 \geq n^\tk, 
\quad \text{and so } \quad
\s_1  \gg\s_2\gg\cdots\gg\s_\dd . \label{sigmal0}
\end{align}

\begin{claim} 
 \label{finalphasesets}
Quite surely,
for all $\ell \in \set{1,\ldots,d}$, 
$m \in \set{0,\ldots,2^\dd-1}$, 
and $a \in \SS \ell m$, 
\begin{align}
\card{\SS \ell m(a)} \geq
 (1-\d)^{2^{\dd-\ellp}-1} \s_{\ell} 
 \label{claimsize}
\end{align}
where $\s_\ell$ is as in \eqref{simpler2y}
and $\d=1/\log^2n$. 
\end{claim}

\begin{proof}
The proof is by induction, proceeding through 
$\ell=d,\ldots,1$ and $m=0,\ldots,2^\ell-1$
in the same sequence as Algorithm \ref{mainalgy}.
At each step in the analysis
we will condition upon a successful history, with no failures 
(i.e., \eqref{claimsize} holds);
the final success probability is the product of the conditional probabilities,
which is at least 1 minus the sum of the conditional failure probabilities.
Conditioning will assure us that the previous set sizes are as desired,
but, as we will see, 
all the triples considered are disjoint
and thus the conditioning has no other effect.
For $\ell=d$,
from Algorithm \ref{mainalgy} Line \ref{Sline1},
for any $m$ and any $a \in N_{\ell,m}$,
$$\card{ \SS \dd m(a)} \sim \Bin(n_0,\pp) . $$
The sizes of the various sets
are independent 
because the trials for one value of $u$
are distinct from those for any other, 
as the first elements of their triples come from 
distinct sets $N_{\ell,m}$ in the partition.
Since $\card{ \SS \dd m(a)} \sim \Bin(n_0,\pp) $,
we have
$\E(\card{\SS \dd m})=\s_\dd$, and the Chernoff bounds imply
\begin{align}
\Pr\brac{\: \card{ \SS \dd m(a)} < (1-\d) \s_\dd \:}
 & \leq
 \exp\set{-\d^2 \s_\dd/3}.
 \label{baseprob}
\end{align}

The analysis of the size of a set $\SS \ell m(a)$ 
in Algorithm \ref{mainalgy} Line \ref{Sline2}
for an internal node, with $\ell < d$,
is slightly more complicated.
For $m$ even, 
a given $b$ belongs to $\SS \ell m (a)$ 
if any of the 
$\card{\SS \ellp {2m} (a)} \: \card{\SS \ellp {2m+1} (b)}$
candidates for $c$ and $d$ leads to a triple $(b,c,d)\in G$.
Conditioning upon success up to this point,
$ \SS \ellp {2m} (a)$ and $\SS \ellp {2m+1} (b)$
each have cardinality satisfying \eqref{claimsize},
so the number of trials for $b$ is at least
\begin{align}
N &= \brac{(1-\d)^{2^{\dd-\ell}-1}\s_{\ellp}}^2 .  \label{Ndef}
\end{align}
The same value $N$ results for $m$ 
odd, so the two cases now continue as one.
These trials for $(b,c,d) \in G$ are independent of all those in the history:
the first index $b$ rotates through distinct choices within node $\ell,m$, 
and $b \in N_{\ell,m}$ assures that first index $b$ was not explored
for any previous node.
By the independence, each trial succeeds with probability $\pp$.
Thus, for a given $b$, at least one trial is successful
with probability at least
$$1- (1-\pp)^N \geq 1-\exp(-\pp N) \geq \pp N (1-\pp N) , $$
where the second inequality relies on $\pp N \leq 1$
(in fact $\pp N=o(1)$, see \eqref{N}).

Considering the $n_0$ candidates for $b$ yields that 
\begin{align}
\card{\SS \ell m (a)} 
 & \sgeq \Bin(n_0, \pp N(1-\pp N)). \label{thebin}
\end{align}

We will shortly need the results of some calculations.
Recall from the hypothesis of Theorem \ref{th2}
that $\dd \leq \e \log_2 \log n$, with $\e<1/2$ 
so that
\begin{align*}
 2^\dd & \leq \log^\e n = o(\log n) .
\end{align*}
Since $\ell$ is playing the role of inductive index here, we use $\lambda$
for a generic equivalent in this paragraph.
In general,
\newcommand{\lambdap}{\lambda+1}
\begin{align}
 \s_\lambda 
  & = (\pp n_0)^{2^{\dd-\lambdap}-1}
   = \parens{ ({n^{\tk-1}}{\log n}) \; \frac{n(1-x/n)}{2^\ddp-2} }
    ^{2^{\dd-\lambdap}-1}
   \leq \parens{ \log n \; n^\tk}^{2^{\dd-\lambdap}-1} , \label{sigmad}
\end{align}and in particular, using
$ \tk (2^\dd-1) = (2^\dd-1) / (2^\ddp-1) \leq 1/2 $,
\begin{equation}\label{kk1}
 \s_1
  \leq
   (\log n)^{2^d-1} n^{\tk (2^\dd-1)}
   \leq (\log n)^{\log^\e n} n^{1/2}
   = n^{1/2+o(1)} .
\end{equation}
 From \eqref{Ndef}, $N \leq \s_\ellp^2$
and, using \eqref{sigmal0} and \eqref{kk1},
\begin{align}
\pp N \leq 
 \pp \s_\ellp^2 =
 \frac1{n_0} \parens{ n_0 \pp\s_\ellp^2 }
 \leq \frac1{n_0} \s_\ell
 \leq \frac1{n_0} \s_1
  \leq n^{-1/2+o(1)} .
  \label{N}
\end{align}

Returning now to \eqref{thebin} and the first part of \eqref{sigmal0}, 
the expectation of the binomial in the RHS is
$$n_0\pp N(1-\pp N)=(1-\pp N)(1-\d)^{2^{\dd-\ellp}-2}\s_\ell,$$
and \eqref{claimsize} holds if the binomial's value is at least
$(1-\d)^{2^{\dd-\ellp}-1}\s_\ell$,
i.e., if it is at least $1-\z := (1-\d)/(1-\pp N)$ times its expectation.

Now \eqref{N} 
gives $\pp N = n^{-1/2+o(1)} \ll 1/\log^2 n = \d$ and
implies that $\z = \d-o(\d)$, implying $\z^2/3 > \d^2/4$. 
Then the Chernoff bound gives us 
\begin{multline}
\Pr\brac{\: \card{ \SS \ell m(a)} < (1-\d)^{2^{\dd-\ellp}-1} \s_\ell \:}
  \leq
 \exp\set{-\z^2(1-\d)^{2^{\dd-\ellp}-2}\s_\ell/3}
 \\ 
  \leq 
 \exp\set{-\d^2(1-\d)^{2^{\dd-\ellp}-2}\s_\ell/4} 
   \leq 
 \exp\set{-\d^2(1-\d)^{2^{\dd}}\s_\ell/4} . \label{midprob}
\end{multline}

The failure probabilities in \eqref{baseprob} and \eqref{midprob}
are both 
$\leq \exp\set{-\d^2(1-\d)^{2^{\dd}}\s_\dd/4}$.
Recall that $2^\dd=o(\log n)$, so 
$$
 \s_\dd =\pp n_0
  \gg n^\tk
  \geq n^{1/\log^{1/2}n}\geq  (\log n)^K
$$
for any fixed $K$.
Using this and
$(1-\d)^{2^\dd} \geq 1-\d 2^\dd = 1-o(1)$,
the failure probability is at most
$$\exp\set{-\frac{\d^2\s_\dd}{4+o(1)}} 
     \leq \exp\set{-\Omega(\d^2 \s_\dd)}  
     \leq \exp\set{-\Omega( (\log n)^{K-4})}
     \leq \exp\set{-\log^2 n} .
$$
The sum of all failure probabilities is thus $\leq 2^\ddp \exp(-\log^2n)$,
and \qs\ there is no failure.
\end{proof}

Claim \ref{finalphasesets} establishes that, \qs,
$\card{\SS 1 1(n)},\card{\SS 1 2(n)} \geq (1-o(1))\s_1$. 
Since $\SS 0 0$ simply consists of triples $(n,a,b) \in G$
with $a \in \SS 1 0(n)$ and $b \in \SS 1 1(n)$,
and no triple with first index $n$ has previously been considered,
conditioned on successful progress of the algorithm to this point,
$\SS 0 0$ is empty with probability
at most 
$(1-\pp)^{\s_1^2/2} \leq \exp\brac{-\pp \s_1^2/2}$. 
But 
\begin{align*}
 \pp \s_1^2 
  &= \frac1{n_0} \s_0
  \geq \frac{2^\ddp-2}{n(1-x/n)} 
       \parens{({n^{\tk-1}}{\log n}) \; \frac{n(1-x/n)}{2^\ddp-2}}^{2^\ddp-1}
   \\ &
  \geq  n^{-1+\tk (2^\ddp-1)} (2^\ddp-2)
     \parens{\frac{\log n}{2^\ddp-1}}^{2^\ddp-1}
  = (2^\ddp-2) \parens{\frac{\log n}{2^\ddp-1}}^{2^\ddp-1}
  \\& \geq 6 (\log n/7)^7 
  \gg \log n.
\end{align*}
The penultimate inequality comes from noting that 
the function 
$(x-1) (a/x)^x$ is log-concave for $x>0$ and so,
over an interval, is minimized at one of the endpoints.
Taking $a=\log n$, putting $\dd=2$ gives $x=7$
for a value of $6(\log n/7)^7$, while putting $\dd$ 
$=\tfrac12 \log_2 \log n$ 
(the right endpoint of a larger interval than is allowed)
gives a larger value.

This shows that the final phase succeeds \qs

\subsection{Execution time}
It is fairly clear that the algorithm runs in polynomial time,
but we will show that it runs in linear time
(remembering that the input is of size $n^3$).
Specifically we will show that
\GreedyPhase takes time $\Theta(n^3)$,
while the other two phases take slightly less:
\MainPhase runs in time $O(n^{3-\frac54\tk})$, and
\FinalPhase in time only $O(n^{\frac52+o(1)})$.

The aspects of each round of \MainPhase that are potentially
time-critical are \SetConstruct, 
and, for each iteration:
finding a healthy element, \MakeTree, and \PoisonPropagate.
The last two cannot dominate.
\MakeTree merely repeatedly finds, for a (non-poisoned) triple,
the (unique) two elements that generated it in \SetConstruct;
thus if we create pointers during \SetConstruct,
this takes time linear in the size of the output.
Since in all iterations of \MakeTree all the output elements are distinct,
the total time is bounded by the size of \SetConstruct's output
and thus by its running time.
Similarly, in \PoisonPropagate, 
the elements poisoned by a given element are a direct function 
of the output of \SetConstruct and can be driven by pointers,
so each poisoning takes unit time.
We contrived to poison at most half the output of \SetConstruct,
so again the time is bounded by \SetConstruct's time.
(Some elements may be poisoned repeatedly,
but in our bound on the sizes of the poisoned sets we 
made the pessimistic assumption that all poisonings were distinct,
so the bound still applies.)

In each round, the time to find a valid element, in 
Algorithm \ref{mainalg2} (\MainPhase) Line \ref{line search},
is $O(\s_1^2) = O(n^2)$.
We will shortly see that this is dominated by the time for \SetConstruct.
The running time for one round of \SetConstruct is of order
$
\sum_{\ell=1}^\dd \sum_{m=1}^{2^\ell} n \s_{\ellp}^2 
= O(n \s_2^2)
$,
since $\s_\ell \gg \s_\ellp$ (see Claim \ref{sizeclaim})
means that the time for each node $(\ell,m)$
dominates the sum of the times for its two children,
so the $\ell=1$ terms dominate the sum.
By \eqref{recurrence},
$n \s_2^2 = (1+o(1)) \frac1\pp \s_1$,
and by \eqref{sigma1}, $\s_1 \leq n^{1-\tk/3}$.
Using $x \leq n^{1-\tk}$ we have
$1/\pp = x^{1+\tk} n^{1-\tk} / \log n$ 
$ \leq n^{2-\tk-\tk^2} 
$,
and thus 
$
 \tfrac1\pp \s_1 \leq n^{3-\tfrac43\tk} n.
$
In all there are $O(\log^2 n)$ rounds,
so by familiar calculations
the total time for all calls to \SetConstruct is
$O(n^{3-\tfrac54\tk})$.

We reason similarly for \FinalPhase.
The time per round to find a valid element 
(Algorithm \ref{FinalPhase2} Line \ref{final2valid})
is $O(\s_1^2) = O(n^2)$, which we shall shortly see
is dominated by the time per round for set construction.
The set construction time (Algorithm \ref{mainalgy}) is of order
$\sum_{\ell=1}^\dd \sum_{m=0}^{2^\ell-1} n_0^2 \s_\ellp^2
 = O(n_0^2 \s_2^2)
 = O(\frac{n_0}{\pp} \s_1)
 $.
By \eqref{simpler} and \eqref{kk1}, this is 
$
O \parens{ \frac{n}{n^{\tk-1}} n^{1/2 +o(1)} }
$.
There are $2^\dd-1 \leq \log^\e n$ rounds,
so the total time is
$n^{\tfrac52-\tk+o(1)}$.

\medskip
This completes the proof of Theorem \ref{th2}.

\section{Difficulties with Axial assignment in dimensions ${\DD>3}$} \label{higherAxial}

Neither \MainPhase nor \FinalPhase successfully carries over to $\DD>3$.
Consider \MainPhase with depth $\dd=2$.
Our algorithm was complicated in order to cope with conditioning,
but ultimately an augmentation,
for $\DD=3$,
relied on choosing $2+4=6$ indices $j,\ldots,s$ from $A$
and 8 indices $\xi_1,\ldots,\xi_8$ from $\bA$
to make $1+2+4=7$ triples cheap.
If ``cheap'' means $\leq \pp$, 
by the first-moment method this can succeed with high probability only if
$n^6 x^8 \pp^7 = \Omega(1)$,
forcing $\pp=\Omega(n^{-6/7} x^{-8/7})$.
When the number of unassigned elements $x$ is $\Theta(1)$, 
this means that a single augmentation costs $\Omega(n^{-6/7})$: 
more than the $O(n^{2-\DD} \log n)$
bound on the \emph{total} assignment cost from Theorem \ref{th1},
but not vastly more.
The equivalent for $\DD=4$,
with a ternary tree, is to choose $3+9=12$ indices from $A$
and 27 from $\bA$ to make $1+3+9=13$ triples cheap, requiring
$\pp=\Omega(n^{-12/13} x^{-27/13})$.
Here, $x=\Theta(1)$ gives $p=\Omega(n^{-12/13})$,
which --- still stuck above $n^{-1}$ ---
is now vastly worse than the $O(n^{2-\DD} \log n)$ upper bound.

\FinalPhase, in its last round, performs a replacement of some
$\kk=2^\dd-2$ assignment $\DD$-tuples
--- which, up to symmetries, we take to be
$\set{(i,i,\ldots,i),1\leq i\leq \kk}$ ---
with $\kk+1$ $\DD$-tuples from 
$\parens{ [\kk] \cup \set{n}}^\DD$.
When $\kk$ is a constant, the probability that there is any new such assignment
of cost $\leq \pp$ 
is at most
$\binom n \kk (\kk!)^{\DD-1} \pp^\kk = O( (n\pp)^\kk)$,
so for the algorithm to have any hope of succeeding
requires
$\pp=\Omega(n^{-1})$.
Again this is satisfactory for $\DD=3$, but for $\DD>3$ it is
vastly worse than the $O(n^{2-\DD} \log n)$ upper bound.

\section{Multi-dimensional Planar assignment} \label{AV}

\subsection{Main theorem} \label{mainPlanar}
Here we give our main theorem for Planar assignment.
\begin{theorem}\label{th3}
The optimal solution value $\ZPl{\DD}$ satisfies the following:
\begin{description}
\item[(a)] $\ZPl{\DD}=\Omega(n^{\DD-2})$ \whp\ for $\DD\geq 3$.
\item[(b)] When $\DD=3$ there is a polynomial-time algorithm that finds a solution with
cost $Z$ where $Z=O(n\log n)$ \whp
\end{description}
\end{theorem}

The theorem is proved in the next two sections.

\subsection{Lower bound}\label{3GA1}
{From} the defining equation \eqref{PP} for Planar assignment,
just attending to constraints of the first type,
for each choice $i_1,\ldots,i_{\DD-1}$ of the first $\DD-1$ coordinates, 
the cheapest element in the line (ranging over the last coordinate)
has distribution $\expdist(n)$,
with expectation $1/n$.
Summing over all $n^{\DD-1}$ such lines, 
the total expected cost is $\geq n^{\DD-2}$.
Each line, independently, has cost $\geq 1/n$ with probability $1/e$,
so there are $\Bin(n^{\DD-1},1/e)$ lines for which this is so,
by a Chernoff bound \qs\ at least $n^{\DD-1}/(2e)$ lines have this property,
and thus \qs\ the total cost is $n^{\DD-2}/(2e)$.

\subsection{Upper bound for $\DD=3$}\label{3GA}
For the upper bound we need a result of Dyer, Frieze and McDiarmid \cite{DFM}. We will not state it in full generality,
instead tailoring its statement to precisely what is needed here.
Suppose that we have a linear program
$$\LP \colon \quad \text{Minimize $c^Tx$ subject to $Ax=b$, $x\geq 0$.}$$
Here $A$ is an $m\times n$ matrix and the cost vector $c=(c_1,\ldots,c_n)$ is a sequence of independent copies of 
$\EX$.
Let $Z_\LP$ denote the minimum of this linear program. Note that
$Z_\LP$
is a random variable. Next let $y$ be \emph{any} feasible solution to $\LP$.
\begin{theorem}[\cite{DFM}]\label{DFM}
\beq{zz}
\E(Z_\LP)\leq m\max_{j=1,\ldots,n}y_j.
\eeq
Furthermore, $Z_\LP$ is at most $1+o(1)$ times the RHS of \eqref{zz}, \whp
\end{theorem}

Our algorithm is based on a characterization of 
a 3-dimensional Planar assignment
as a collection of 2-dimensional Planar assignments.
(We think of these as $(D-1)$-dimensional Planar assignments,
but in 2 dimensions Planar and Axial assignment are the same.)
for For generalization to dimensions $D>3$ we think of these
for purposes o
Fix a 1-coordinate $i$.
Then for every 2-coordinate $j$, 
by the first constraint type of \eqref{PP}, the line $(i,j,*)$
must contain exactly one assignment element,
and by the third constraint type,
the same is true for every 3-coordinate $k$ and line $(i,*,k)$.
Thus, the assignment elements in plane $(i,*,*)$
define a 2-dimensional assignment. 

Of course there are constraints among these 2-dimensional assignments. 
By the second constraint type,
for $i \neq i'$, the elements $(i,j,k)$ and $(i',j,k)$
cannot both be selected; 
that is, 
\begin{align}  \label{alt3}
 (\forall j,k) \sum_i X_{i,j,k} \leq 1 . 
\end{align}
This condition is sufficient as well as necessary:
any collection of 2-dimensional assignments satisfying it
is a 3-dimensional assignment.
This is clear because the collection of 2-dimensional assignments has 
$n^2$ elements, 
thus $\sum_{j,k} \sum_i X_{i,j,k} = n^2$,
thus if \eqref{alt3} is always satisfied, 
it must in every case be satisfied with equality.
In that case the selection satisfies the second constraint type of \eqref{PP}
as well as the first and third, and is a 3-dimensional assignment.

Now consider the following greedy-type algorithm.
First, find a minimum 2-dimensional assignment for
1-plane $i=1$. 
For each element $(j,k)$ selected,
remove $(j,k)$ from consideration for 2-dimensional assignments
for all 1-planes $i'$ with $i'>i$.
Find a minimum assignment for 1-plane $i=2$ 
with this restriction,
remove these elements from consideration, and so on.

\newcommand{\GreedyPlanarT}{\textsc{GreedyPlanar3D}}
\newcommand{\Ma}{\Xi}

\begin{algorithm}[H]
\caption{\GreedyPlanarT: construct a 3-dimensional Planar assignment}
\label{greedy3p}
\begin{algorithmic}[1]
\FOR{$i=1,\ldots,n$}
\STATE Let $H_i=K_{n,n}\setminus (\Ma_1\cup \cdots \cup \Ma_{i-1})$
\STATE Give each edge $(j,k)\in E(H_i)$ cost $\Mijk$
\STATE Let $\Ma_i$ be a minimum cost matching of $H_i$ 
\ENDFOR
\STATE Return $\Ma_1,\ldots,\Ma_n$
\end{algorithmic}
\end{algorithm}

The output $\Ma_1,\ldots,\Ma_n$ defines a set of triples $T=\set{(i,j,k)
\colon (j,k)\in \Ma_i}$
that by the previous discussion is a 3-dimensional Planar assignment.
Writing $Z_i$ for the cost of matching $\Ma_i$, we claim that
\beq{ax}
\E(Z_i)\leq \frac{2n}{n-i+1}.
\eeq
For this we apply Theorem \ref{DFM} to the following linear program,
which we note
always has an integer optimum solution:
\begin{align*}
\text{Minimize} \sum_{(j,k)\in E(H_i)} \Mijk \; x_{j,k} 
\quad & \text{ subject to} 
\\
\sum_{k \colon (j,k)\in E(H_i)}x_{j,k}=1, \qquad & j=1,\ldots,n\\
\sum_{j \colon (j,k)\in E(H_i)}x_{j,k}=1, \qquad & k=1,\ldots,n\\
x_{j,k}\geq 0, \qquad & j,k=1,\ldots,n .
\end{align*}
We note that there are $2n$ constraints and that $x_{j,k}=1/(n-i+1)$ is a
feasible solution.
By Theorem~\ref{DFM}, this implies \eqref{ax} and the upper bound in Theorem \ref{th3} for the
case $D=3$.

\section{Difficulties with Planar assignment in dimensions ${\DD>3}$} 
\label{higherPlanar}
The previous section's characterization of a 3-dimensional assignment
as a collection of $n$ 2-dimensional assignments extends:
a $\DD$-dimensional assignment is a collection of $n$ 
$(\DD-1)$-dimensional assignments,
one for each 1-index $i_1$,
satisfying the property that if $(i_2,\ldots,i_n)$ belongs to assignment $i_1$
then it belongs to no other assignment $i'_1$, i.e.,
$(\forall i_2,\ldots,i_n) \sum_{i_1} X_{i_1,\ldots,i_n} \leq 1$.
As before, this constraint is necessary 
(by the defining constraint on lines for coordinate~1)
and also sufficient:
by the same counting argument as before, if the constraint is 
always satisfied then it is always satisfied with equality,
which thus implies the assignment-defining constraint 
for lines along coordinate 1.
The constraints for lines along other coordinates are satisfied 
since each is included in some $(\DD-1)$-dimensional assignment.

The corresponding algorithm we have in mind is a generalization 
of Algorithm \ref{greedy3p}:
repeatedly solve $(\DD-1)$-dimensional instances
and remove their elements from future consideration.
Each lower-dimensional instance is solved by a recursive call to the same
algorithm, until dimension 2 is reached and the instance is solved as an LP.
Note that all but the first of the $(\DD-1)$-dimensional instances 
has only a subset of the elements available for assignment: 
it is a matching problem on a $(\DD-1)$-partite hypergraph,
but not the complete $(\DD-1)$-partite hypergraph.
That is, the generalized algorithm should solve ``matching'' problems,
not just ``assignment'' problems.

The difficulty is that such problems do not always admit a solution,
and even when we start with the complete hypergraph ---
an assignment problem, which does always have a solution ---
we may generate subinstances that do not.

A small example is shown below,
with $\DD=4$ and $n=4$.
Suppose that in the first 3-dimensional sub-array considered 
(shown as four 2-dimensional arrays),
the elements selected are those indicated by 1s in the table below;
this will happen, for example, if the corresponding cost elements 
are small and the others are large.
\begin{verbatim}
 0 1 0 0   0 0 0 1   0 0 1 0   1 0 0 0 
 0 0 1 0   1 0 0 0   0 1 0 0   0 0 0 1 
 1 0 0 0   0 1 0 0   0 0 0 1   0 0 1 0 
 0 0 0 1   0 0 1 0   1 0 0 0   0 1 0 0 
\end{verbatim}

Then, in the second 3-dimensional sub-array, perhaps the following
selection gives the cheapest assignments
for the first three 2-dimensional instances solved.

\begin{verbatim}
 0 0 1 0   0 1 0 0   0 0 0 1
 0 1 0 0   0 0 1 0   1 0 0 0
 0 0 0 1   1 0 0 0   0 1 0 0
 1 0 0 0   0 0 0 1   0 0 1 0
\end{verbatim}
There is no way to complete this 3-dimensional array:
the last 2-dimensional instance has no solution.
We check that it is impossible to select any element in its first row.
The first row's first element is blocked by the 
2-dimensional array above it
(the 4th 2-dimensional matching comprising the first 3-dimensional matching)
and the row's remaining elements are blocked 
by the 2-dimensional arrays to the left of it
(respectively, by
the 2nd, 1st, and 3rd 2-dimensional matchings for the 
second 3-dimensional matching problem).

This example shows that in dimensions $\DD>3$,
the obvious generalization of our 3-dimensional Planar assignment algorithm 
may fail to return any solution, regardless of cost.

\section{Conclusions}
For the 2-dimensional random assignment problem,
we know the limiting expected cost, and a given instance
can be solved in polynomial time.
As noted in the introduction, much less is known about multi-dimensional
assignment problems, and as far as we are aware,
very little was known about polynomial-time algorithms solving these problems
well on average, especially for random costs with a density $f(x)$ that is close to a constant for $x$ 
close to zero.

For the 3-dimensional Axial assignment problem,
we give an upper bound within $n^{o(1)}$
of the obvious $\Omega(1/n)$ lower bound,
as a trivial application of 
an extension of
a result of
Johansson, Kahn and Vu \cite{JKV}.
Our main result is an algorithm that constructs \whp\
constructs a solution of cost $n^{o(1)}$,
although the cost bound here is not as small as that coming from \cite{JKV}.

For the 3-dimensional Planar assignment problem,
we prove an upper bound
within a logarithmic factor of the obvious $\Omega(n)$ lower bound
(likely the true answer),
by analyzing a simple and fast greedy algorithm.

Neither result extends to $\DD > 3$.
We are left with open questions including these:
\begin{description}
\item[P1] What are the growth rates of
$\E[\ZAx{\DD}]$ and $\E[\ZPl{\DD}]$ for $\DD\geq 3$?
\item[P2] Are there asymptotically optimal,
polynomial-time algorithms for solving these problems when $\DD\geq 3$?
\item[P3]
For $\DD>3$, are there polynomial-time algorithms yielding solutions within
logarithmic or $n^{o(1)}$ factors for Planar and Axial assignment problems
(as we have given for $\DD=3$)?
\end{description}

% \bibliographystyle{amsalpha}
% \bibliography{dass}
% \end{document}

\end{document}